\documentclass[a4paper,11pt]{article}
\usepackage{graphics,color}
\usepackage{amsmath}
\usepackage{amssymb}
\usepackage{mathrsfs}
\usepackage{cite}
\usepackage{verbatim}
\usepackage{float}
\usepackage{graphicx}
\usepackage{amsthm}
\usepackage{textcomp}
\usepackage{subfig}
\usepackage{hyperref}

\newif\ifdraft
\draftfalse 
\newcommand{\Aone}{\areaonecod}
\newcommand{\area}{\mathcal A}
\newcommand{\areaonecod}{\overline{\mathcal A}}
\newcommand{\axialcoordofcylinder}{{w_1}}

\newcommand{\BV}{{\rm BV}}

\newcommand{\dirdatum}{\varphi}
\newcommand{\dts}{\Sigma}
\newcommand{\dtsm}{\Sigma^-}
\newcommand{\dtsn}{\Sigma_n}
\newcommand{\dtsp}{\Sigma^+}
\newcommand{\dtsnm}{\Sigma_n^-}
\newcommand{\dtsnp}{\Sigma_n^+}
\newcommand{\doubledrectangle} {R_{2\longR}}
\newcommand{\duevirgolan}{{2,n}}

\def\eps{\varepsilon}
\newcommand\extpsi{\widehat \psi}

\newcommand{\FB}{{\mathcal F}_{2\longR}}
\newcommand{\FBl}{{\mathcal F}_\longR}

\newcommand{\graphh}{G_ \h}
\newcommand\grad{\nabla}

\newcommand{\h}{h}

\newcommand{\hn}{{\h_n}}
\newcommand{\Hone}{\mathcal{H}^1}
\newcommand{\Htwo}{\mathcal{H}^2}
\newcommand{\Hricc}{\mathcal H_\longR}
\newcommand{\Hspace}{\mathcal H_{2l}}

\newcommand{\imenounovirgolan}{{i-1,n}}
\newcommand{\ipiuunovirgolan}{{i+1,n}}
\newcommand{\ivirgolan}{{i,n}}

\newcommand{\jump}[1]{\text{{\rm \textlbrackdbl}}{#1}\text{{\rm \textrbrackdbl}}}

\newcommand{\Lh}{L_\h}
\newcommand{\Lone}{L^1}
\newcommand{\longR}{l}

\newcommand{\modh}{h^\star}
\newcommand{\modpsi}{\psi^\star}
\newcommand{\modhbis}{h^\#}
\newcommand{\modpsibis}{\psi^\#}

\newcommand{\nada}[1]{}
\def\NN{\mathbb{N}}

\newcommand{\Omegah}{\subgraph_\h}
\newcommand{\Omegahn}{\subgraph_\hn}

\newcommand{\oldDom}{X_\longR}
\newcommand{\Om}{\Omega}
\newcommand{\openset}{O}

\newcommand{\partialbar}{ \partial_D}
\newcommand{\pointP}{\mathcal P}
\newcommand{\pointQ}{\mathcal Q}
\newcommand{\projthree}{\pi_3}
\newcommand{\psione}{\psi}

\newcommand{\psionen}{\psione_n}

\newcommand{\puntone}{Z}
\newcommand{\puntino}{z}

\newcommand{\R}{\mathbb{R}}

\newcommand{\relarea}{\overline \area}

\newcommand{\res}{\mathop{\hbox{\vrule height 7pt width 0.5pt depth 0pt
\vrule height 0.5pt width 6pt depth 0pt}}\nolimits}

\newcommand{\scoord}{{w_2}}
\newcommand{\scalarfunction}{\psi}
\newcommand{\sourcedisk}{{\rm B}}
\newcommand{\subgraph}{SG}
\newcommand{\subgraphh}{SG_h}
\def\supp{\,{\rm supp \ }}

\newcommand{\tcoord}{{w_1}}
\newcommand{\terzopunto}{t_*}

\newcommand{\unovirgolan}{{1,n}}
\newcommand{\unitdisc}{{\rm D}}

\newcommand{\vortexmap}{u}

\newcommand{\Wspace}{X_\longR^{\rm conv}}

\numberwithin{equation}{section}
\mathchardef\emptyset="001F

\newtheorem{theorem}{Theorem}[section]
\newtheorem{definition}[theorem]{Definition}
\newtheorem{prop}[theorem]{Proposition}
\newtheorem{cor}[theorem]{Corollary}
\newtheorem{lemma}[theorem]{Lemma}

\theoremstyle{definition}
\newtheorem{remark}[theorem]{Remark}
\newtheorem{Remark}[theorem]{Remark}

\title{Relaxation of the area of the vortex map: a non-parametric Plateau problem for a catenoid containing a segment}
\author{
Giovanni Bellettini\footnote{
Dipartimento di Ingegneria dell'Informazione e Scienze Matematiche, Universit\`a di Siena, 53100 Siena, Italy,
and International Centre for Theoretical Physics ICTP,
Mathematics Section, 34151 Trieste, Italy.
E-mail: giovanni.bellettini@unisi.it
                      }\and
Alaa Elshorbagy\footnote{
Technische Universit\"at Dortmund, 
Fakult\"at f\"ur Mathematik, 
44227 Dortmund, Germany. E-mail: elshorbagy.alaa1@gmail.com 
                         }
                          \and
Riccardo Scala\footnote{ 
Dipartimento di Ingegneria dell'Informazione e Scienze Matematiche, Universit\`a di Siena, 53100 Siena, Italy.
E-mail: riccardo.scala@unisi.it
}
}

\begin{document}

\maketitle

\begin{abstract}
Motivated by the study of the 
non-parametric
area $\mathcal A$ of the graph of the vortex map $\vortexmap$
(a two-codimensional singular surface in $\R^4$) 
 over
 the disc $\Omega \subset \R^2$ of radius $\longR$,
we perform a careful analysis of the singular part
of the relaxation of $\mathcal A$ computed at $\vortexmap$. 
The precise description is given in terms of a area-minimizing
surface in a vertical copy of $\R^3 \subset \R^4$, which is a sort of ``catenoid'' containing
a segment corresponding to a radius of $\Omega$. The problem
involves an area-minimization with a free boundary part; 
several boundary regularity properties of the minimizer
are inspected.
\end{abstract}

\noindent {\bf Key words:}~~Relaxation, non-parametric
minimal surfaces, Plateau problem, area functional in codimension two.

\vspace{2mm}

\noindent {\bf AMS (MOS) subject clas\-si\-fi\-ca\-tion:} 
49Q15, 49Q20, 49J45, 58E12.

%%%%%%%%%%%%%%%%%%%%%%%%%%%%%%%%%%%%%%%%%%%%%%%%%%%%%%%%%%%%%%%%%%%%%%%%%
\section{Introduction}\label{sec:introduction}
%%%%%%%%%%%%%%%%%%%%%%%%%%%%%%%%%%%%%%%%%%%%%%%%%%%%%%%%%%%%%%%%%%%%%%%%%
Let $l >0$, and 
$\Om = \sourcedisk_\longR(0) = \{x = (x_1, x_2)\in \R^2 : 0
\leq \vert  x\vert <l\} \subset \R^2$. 
Consider 
the vortex map
\begin{equation}\label{vortexmapdef_intro}
	\vortexmap(x):=\frac{x}{|x|}, \qquad x \in \Om \setminus \{0\} \ 
\end{equation} 
The recent results of \cite{BES1,BES3} provide a formula, obtained via a relaxation 
procedure, for the area of the graph of $u$, a singular 
two-dimensional surface in $\R^4$. More specifically, 
consider the classical expression
 \begin{equation*}
 	\area(v,\Omega)
 	=\int_\Omega\sqrt{1+|\nabla v(x_1,x_2)|^2+|Jv(x_1,x_2)|^2}~dx_1dx_2
\qquad 
\forall v \in C^1(\Om; \R^2),
 \end{equation*}
 where $\nabla v$ is the gradient of $v$, a $2\times2$ matrix, 
 $\vert \nabla v\vert^2$ is the sum of the squares of all
 elements of $\nabla v$, 
 and $Jv$ is the Jacobian determinant of $v$, {\it i.e.}, the determinant of $\nabla v$.
Namely, $\area(v,\Omega)$ is  
the area  of the graph of the smooth map $v$. Now, 
denote by 
\begin{equation}\label{eq:area_relax}
	\relarea(v,\Omega):=\inf
	\Big\{\liminf_{k\rightarrow +\infty}\area(v_k,\Omega)\Big\} 
\qquad \forall v \in 
L^1(\Om, \R^2), 
\end{equation}
the sequential relaxation of  $\area(\cdot,\Omega)$ in 
the $L^1$-convergence.
The infimum in \eqref{eq:area_relax} is computed 
over all sequences of maps $v_k\in C^1(\Omega,\R^2)$ 
approaching $v$ in $L^1(\Omega,\R^2)$.
Despite $\area(\cdot, \Om) = \relarea (\cdot, \Om)$ on 
$C^1(\Om, \R^2)$,
the computation of $\relarea(v,\Omega)$ for  $v \notin C^1(\Om, \R^2)$ 
appears to be
at the moment out of reach, due essentially
to highly nonlocal phenomena created by the mutual interaction 
of the singularities of the map $v$, and also
by the interaction of such singularities with the boundary
of $\Om$ \cite{BSS,Scala:19,ScSc}. 
However,  in the case the map $v$ equals the vortex map $u$, 
the following  is proven in 
\cite{BES1,BES3}:
		\begin{align}\label{value_main}
\relarea (u,\sourcedisk_\longR(0))
		=\int_{\sourcedisk_\longR(0)}\sqrt{1+|\nabla u|^2}dx+
2
\inf_{(h,\psi) \in 
	X_l}	
\mathcal F_l(h,\psi),
	\end{align}
showing the presence of the interesting singular term
\begin{equation}\label{eq:interesting_singular_term}
2
\inf_{(h,\psi) \in 
	X_l}	
\mathcal F_l(h,\psi),
\end{equation}
 expressed as
an infimum of an appropriate functional 
$\mathcal F_l(h,\psi)$, that we are going to describe, and 
that will be the argument of the present paper.
If we set $R_l = (0,l) \times (-1,1)$, and 

\begin{equation}\label{eq:X_l_intro}
\begin{aligned}
	\oldDom
	:=& 
\Big\{(h,\psi):
	h\in L^\infty([0,\longR];[-1,1]), \psi \in \BV(R_l;[0,1]),
\\
& \quad \quad \psi=0 
~{\rm in ~}
SG_h := \{(w_1,w_2) \in R_l: w_2<h(w_1)\} 
	\Big\},
\end{aligned}
\end{equation}
 the functional $\FBl:\oldDom\rightarrow [0,+\infty)$ in \eqref{value_main}
is defined as
\begin{equation}\label{eq:F_l}
	\FBl(h,\psi)
:=\areaonecod
	(\psi, R_l)-\mathcal H^2(R_l\setminus SG_h)+\int_{\partial_DR_l}|\psi-\varphi|~d\mathcal H^1+\int_{(0,\longR)\times \{1\}}|\psi|~d\mathcal H^1.
\end{equation}
Since the function $\psi$ takes scalar values, 
the expression $
\areaonecod (\psi, B)$ 
of the area 
of its graph in any Borel set $B 
\subseteq R_l$ (i.e., the $L^1$-relaxation of 
the localized area $\area(\cdot, B)$ defined on $C^1(B)$
or on $W^{1,1}(B)$) has the well-known expression 
\begin{equation}\label{eq:relaxed_area_one_cod}
\areaonecod (\psi, B) = 
\int_{B} \sqrt{1 + \vert\grad^a \psi\vert^2}~dx + 
\vert D^s \psi\vert(B),
\end{equation}
with $\grad^a\psi$ the absolutely continuous and $D^s\psi$ the
singular part of the measure $D \psi$ (the distributional gradient of $\psi$). Also, the subgraph of $\psi$, and the trace 
of $\psi$ on the boundary of $R_l$, are well-defined 
by classical results 
on $BV$-functions, and so \eqref{eq:F_l} is well-defined.
Here the Dirichlet part $\partial_D R_l$ of $\partial R_l$ 
is given by two of the four sides, 
$\partial_D R_l = (\{0\} \times [-1,1])
\cup ((0,\longR) \times \{-1\})$, and the Dirichlet
boundary condition $\varphi: \partial_D R_{\longR} \rightarrow [0,1]$, dictated by the geometry of the vortex
map, is given by
$\varphi(t,s)
:= \sqrt{1-s^2}$ 
if $(t,s) \in \{0\} \times [-1,1]$
and $\varphi(t,s):=0$ if $(t,s) \in (0,\longR) \times \{-1\}$; see Fig. \ref{fig:graphofphi}. 
The multiplicative factor $2$ in front of the infimum in \eqref{eq:interesting_singular_term}
is due to the fact that we find 
convenient to describe the singular term, as we shall see,
using a non-parametric Plateau problem, while the original relaxation 
for the vortex map takes into account also the area contribution of the reflected surface over
the horizontal plane.

The aim of this paper is a careful analysis of the functional
$\FBl$ and of its domain $\oldDom$, and the precise
computation of \eqref{eq:interesting_singular_term} and of its minimizers;
our results
shed 	
light on the geometric meaning of the two-codimensional
question posed by the computation of $\relarea(u, \Om)$, and on 
its nonlocality with respect to $\Omega$. 
As we shall see, a non-parametric Plateau 
problem with partial free boundary pops-up;  
its solution turns out to be 
an area-minimizing surface of disc-type, having trace
half of a ``catenoid'' constrained to contain a segment (a radius of $\Om$). The intuitive reason for this is the following:
if we pick in the source domain 
$\Omega$ a generic circle surrounding the origin (the 
singular point of $u$), and we look 
at the corresponding values over it of a minimizing sequence $(v_k)$ of {\it smooth} maps approximating $u$  in $L^1(\Om, \R^2)$,
we can reasonably expect that these values almost fill a copy of $\mathbb S^1$ in the 
target plane.
The question then becomes: how can 
$v_k$ approximate $u$ in such a way that the area (in $\R^4$) of their graphs
is as small as possible? 
The answer given here is that
the best way is to construct, in a ``vertical'' copy of $\R^3$
obtained by dropping from $\R^4$ an appropriate coordinate,  half of a ``catenoid'' (together with 
its mirror reflection with respect to the horizonal plane), hinged exactly
on a radius of $\Om$; 
morally, this produce a sort of 
phantom optimal ``catenoidal'' tube joining the singularity
with the boundary of $\Om$.
As we shall see, a technical simplification in the formulation of the problem will consist
in doubling the rectangle $R_l$, in order to 
make the boundary of $\Om$ disappear in some sense, creating instead
another point singularity at distance $2l$ from the origin (a sort 
of image charge), where 
assigning a boundary condition similar to the one given above the 
origin.

To introduce the correct setting we need to fix some 
notation. 
Denote $R_{2\longR}:=(0,2l)\times (-1,1)$ the doubled rectangle,
and let $\partial_D\doubledrectangle
:=(\{0,2l\}\times[-1,1])\cup((0,2l)\times \{-1\})$,
now consisting of three sides of $\partial \doubledrectangle$.
We also introduce the map
$\varphi: \R^2 \rightarrow [0,1]$ as
\begin{align}\label{eq:varphi}
\varphi(w_1,w_2):=\begin{cases}
	\sqrt{1-w_2^2}&\text{if }|w_2|\leq 1,\\
	0&\text{otherwise,}
\end{cases}
\end{align}
which we will employ as Dirichlet boundary datum.
Let 
\begin{align*}
&S_{2l}:=\{\sigma:(0,1)\rightarrow \doubledrectangle \; \text{ Lipschitz and injective,} \;\sigma(0^+)=(0,1),\;\sigma(1^-)=(2l,1)\},\\
	&
\mathcal X_{D,\varphi}
:=\{\psi\in W^{1,1}(\doubledrectangle):\psi=\varphi\text{ on }\partial_D\doubledrectangle
\},
\end{align*}
where we have noted $\sigma( x_0^\pm):=\lim_{x\rightarrow x_0^\pm}\sigma(x)$. For any $\sigma \in S_{2l}$ we denote by $A_\sigma$ the open planar region 
enclosed between $\sigma((0,1))$ and $\partial_D \doubledrectangle$. Let 
us first consider the following minimum problem:
\begin{align}\label{plateau_intro1}
	\inf\{ \mathcal A(\psi,A_\sigma):(\sigma,\psi)\in S_{2l}\times 
	\mathcal X_{D,\varphi},\;\psi=0\text{ on }\sigma((0,1))\}.
\end{align} 
Since this problem in general has not a minimizer,  we 
need a relaxed formulation. 
However we prefer not to directly relax problem \eqref{plateau_intro1}, but 
instead we
reduce to a cartesian setting, where 
the free-boundary curve $\sigma$ becomes
 the graph of a function $h$ defined on $(0,2l)$  and $A_\sigma$
its subgraph $SG_h$: namely, we introduce
 \begin{align*}
	&
\widetilde {\mathcal H}_{2l}:=\{
	h:[0,2l]\rightarrow [-1,1] \; {\rm continuous},~ h(0)=h(2l)=1
	\},
\end{align*}
and for any $h\in \widetilde {\mathcal H}_{2l}$ set
$G_h:=\{(t,s)\in \doubledrectangle:s=h(t)\}$ and 
$SG_h:=\{(t,s)\in \doubledrectangle:s
< h(t)\}$ (see Fig. \ref{fig:graphofphi}). We then consider
the 
minimum problem
\begin{align}\label{eq:plateau_tilde_intro}
	\inf\{ \mathcal A(\psi,SG_h):(h,\psi)\in \widetilde {\mathcal H}_{2l}\times 
	\mathcal X_{D,\varphi},\;\psi=0\text{ on }G_h\}.
\end{align}
Focusing on this, we will 
show that
it is sufficient to restric attention to convex functions $h$, and thus we 
introduce a sort of relaxed functional, namely
	\begin{equation}
	\FB(\h,\psione):=\Aone(\psione;\doubledrectangle)-
	\Htwo(\doubledrectangle \setminus \Omegah)+\int_{\partialbar\doubledrectangle}  
	\vert \psione  -\dirdatum\vert d \Hone 
	+ \int_{
		\partial \doubledrectangle \setminus \partialbar\doubledrectangle
	}  \vert \psione \vert 
	~d \Hone,
	\label{eq:F_2l}
\end{equation} 
defined on  
	\begin{equation}\label{eq:X_2l_conv}
	X_{2\longR}^{\rm conv}
	:=
	\left\{(h,\psi): h \in \Hspace, \psi\in BV(\doubledrectangle, [0,1]),\psi = 0 \ 
	{\rm on}~ R_l \setminus SG_h
\right\}, 
	\end{equation}
	where
\begin{equation}\label{eq:space_of_h_in_the_doubled_interval}
\Hspace
=\big \{\h : [0,2 \longR] \to [-1,1], 
	~ \h {\rm~ convex}, ~
	\h(\axialcoordofcylinder)=\h(2\longR-\axialcoordofcylinder) 
	~\forall \axialcoordofcylinder \in [0,2\longR]
	\big \}.
\end{equation}
 Notice that,
with respect to $\widetilde {\mathcal H}_{2l}
\times \mathcal X_{D,\varphi}$, the class $X_{2l}^{\textrm{conv}}$ 
is obtained 
specializing the choice of $h$ but
generalizing the choice of $\psi$.
 
Then we will prove (Theorem \ref{teo:main1}) that 
\begin{align}\label{plateau_intro2}
\inf_{(h,\psi) \in 
	X_{2\longR}^{\rm conv}}	\FB(h,\psi)=\inf\{ \mathcal A(\psi,SG_h):(h,\psi)\in 
\widetilde {\mathcal H}_{2l}\times \mathcal X_{D,\varphi},
\;\psi=0\text{ on }G_h\},
\end{align} 
and that, for $\longR$ large enough, the infimum on the right-hand side is not attained in $\widetilde {\mathcal H}_{2l}\times \mathcal X_{D,\varphi}$ and equals $\pi$. 
Instead a minimizer $(\overline h,\overline \psi)
\in \widetilde {\mathcal H}_{2l}\times \mathcal X_{D,\varphi}$ exists for $\longR$ small, 
and $\overline\psi$ is 
real analytic 
in the interior of $SG_{\overline h}$;
furthermore, we show that 
$\overline h$ is smooth and convex, and 
$\overline\psi$ has vanishing 
trace on  $G_{\overline h}$. 

Thus, problem 
\eqref{eq:plateau_tilde_intro} and the minimization of the functional in \eqref{eq:F_2l}
are related through 
\eqref{plateau_intro2}; 
now, let us see how they
are related with \eqref{plateau_intro1}.
An equivalent formulation of \eqref{plateau_intro2}
is the following: 
for $\sigma\in S_{2l}$,
take the closed curve 
$\Gamma\subset\R^3$ defined by glueing the trace of $\sigma$ with the graph of 
$\varphi$ over $\partial_D \doubledrectangle$. 
We can then consider an  area-minimizing disc $\Sigma^+$ spanning $\Gamma$, solution of the 
classical parametric Plateau problem. 
If
$\sigma([0,1])$ 
is the graph of a function $h\in\widetilde{\mathcal H}_{2l}$, then
\begin{align}\label{min_problem_plateau}
	\inf\{ \mathcal A(\psi,SG_h):(h,\psi)\in \widetilde {\mathcal H}_{2l}\times 
	\mathcal X_{D,\varphi},\;\psi=0\text{ on }G_h\}=\inf \mathcal H^2(\Sigma^+),
\end{align} 
where the infimum on the right-hand side is computed over the set of all such 
curves $\sigma$ (see Corollary \ref{main_cor}). In turn, we will
prove that 
\begin{align}\label{equiv_1}
\inf \mathcal H^2(\Sigma^+)=	\inf\{ \mathcal A(\psi,A_\sigma):(\sigma,\psi)\in S_{2l}\times 
\mathcal X_{D,\varphi},\;\psi=0\text{ on }\sigma((0,1))\},
\end{align}
so that the original problem \eqref{plateau_intro1} is equivalent to 
\eqref{eq:plateau_tilde_intro}.

Indeed, for $\longR$ sufficiently
small, say $l
\in (0,l_0)$, the infimum on the left-hand side of \eqref{equiv_1} is attained by a disc-type surface $\Sigma^+$, and $\sigma([0,1])$ coincides with the graph of a smooth convex 
function $h\in \widetilde {\mathcal H}_{2l}$. Also, the surface $\Sigma_+$ is cartesian, i.e. is the graph of a function $\psi\in \mathcal X_{D,\varphi}$. On the contrary, for $l\geq l_0$, $\Sigma^+$ is degenerate, 
in the sense that if $\sigma_n((0,1))\subset \doubledrectangle$ is the 
free-boundary of $\Sigma^+_n$, where $(\Sigma_n^+)$ is 
a minimizing sequence of discs for the Plateau problem, 
then $\sigma_n((0,1))$ converges to 
$\partial_D \doubledrectangle$ and $\Sigma^+_n$ converges to 
two distinct half-circles of radius $1$, whose total area is $\pi$.

We do not know the explicit value of the threshold $l_0$.
However,
it is clear that $l_0>\frac12$ 
(see the discussion at the end of Section \ref{subsec:a_Plateau_problem_for_a_self-intersecting_boundary_space_curve} and Remark \ref{rem:threshold}).
Furthermore, if we double the surface $\Sigma^+$ by considering 
its symmetric with respect to the plane containing $\doubledrectangle$, and then 
taking the union $\Sigma$ of these two area-minimizing surfaces, 
it turns out that $\Sigma$ solves a non-standard Plateau problem, spanning a nonsimple curve which shows self-intersections 
(this is the union of $\Gamma$ with its symmetric with 
respect to $\doubledrectangle$, the obtained curve is the union of two circles connected by a segment, see Section \ref{subsec:a_Plateau_problem_for_a_self-intersecting_boundary_space_curve} and Fig. \ref{fig:curvagamma}). 
Again, the obtained
area-minimizing surface is a sort of catenoid forced to contain a segment
(see Fig. \ref{elliss_1}, left) for $\longR$ small, and two distinct discs spanning the two circles for $\longR$ large (Fig. \ref{elliss_1}, right).  

A related analysis of a general geometric setting for 
constrained non-parametric Plateau problems 
can be found in \cite{BMS}. Notice however that in \cite{BMS} it is assumed 
positiveness of the boundary datum,
while in our case it is crucial that 
$\varphi=0$ on $\{-1\}\times (0,2l)$; 
it must be observed that the vanishing of $\varphi$ a some points
is source
of a number of difficulties.
Furthermore we need here to relate the Plateau problem 
to the functional $\mathcal F_l$ in \eqref{eq:F_l}, analysis which is missing in \cite{BMS}. Indeed, in the special setting of the
present paper, some more precise description of solutions is necessary. Specifically (see Corollary \ref{cor_reg}) every solution will be continuous and null on its free boundary curve.

Before stating our main results, it is worth to point out that the 
solution surface that we obtain 
is related to the problem of relaxation of area functional in codimension $>1$. Indeed, the restriction of $\Sigma$ to the set $\overline B_1(0)\times [0,l]$ is a suitable projection in $\R^3$ of the vertical part of an optimal cartesian current with underlying map the vortex 
map (see the introduction of \cite{BES3}).

In Proposition \ref{modifications_of_h} below, we first prove that the infimum of 
$\FBl$ on the class 
$\oldDom$ can be equivalently computed on 
pairs $(h,\psi)\in \Wspace$, where the function $h$ is assumed to be convex, nonincreasing and continuous on $[0,\longR]$. Then, using a symmetry argument, it is easily seen that 
$$2\inf_{(h,\psi)\in \Wspace}\FBl(h,\psi)=   \inf_{(h,\psi) \in 
	X_{2\longR}^{\rm conv}}	\FB(h,\psi).$$
The main results of this paper are contained in the following
two theorems.

\begin{theorem}\label{teo:main1}
	There exists a solution of 
	\begin{equation}\label{eq:B_intro}
\min \big \{ \FB(\h,\psione):  (\h,\psione) \in  X_{2\longR}^{\rm conv}
		\big \}.
	\end{equation}
	Moreover, any 
minimizing pair $(h,\psi)$ in \eqref{eq:B_intro}
satisfies the following properties:
\begin{itemize}
\item[(1)]
$\psi$ is symmetric with respect to 
	$\{\axialcoordofcylinder=\longR\} \cap \doubledrectangle$;
\item[(2)]
 If $h$ is not identically $-1$, then 
	\begin{itemize}
		\item[(2i)] $h\in 
C^0([0,2\longR])$ and is analytic in 
$(0,2\longR)$,
 $h(0)=1=h(2\longR)$, and 
		$h >-1$ in $(0,2\longR)$;
		\item[(2ii)] $\psi$ is analytic 
		and strictly positive
		in $\Omegah$; 
		\item[(2iii)] $\psi$ 
		is continuous up to the boundary of $\Omegah$, 
		and attains the boundary condition, \textit{i.e.}, 
for $(\tcoord, \scoord)
		\in 
		\partial SG_{h}$,  
		\begin{equation}\label{eq:psionestarboundary_intro}
			\psi
(\axialcoordofcylinder,\scoord)=\begin{cases}
				0& \text{ if}\quad \scoord=-1 
\\ 
				0& \text{ if}\quad 
\scoord=h(\axialcoordofcylinder),\\
				\sqrt{1-\scoord^2}&\text{ if}\quad \axialcoordofcylinder=0 \text{ or } \axialcoordofcylinder=2\longR,
			\end{cases} 
		\end{equation}
		hence 
		\begin{equation}\label{eq:FAequalareaofhstar_intro}
			\FB(h,\psi)=\mathcal A(\psi, \Omegah);
		\end{equation}
		\item[(2iv)] 
			$\psi < \varphi {\rm ~in}~ \doubledrectangle$.
	\end{itemize}
	\end{itemize}
\end{theorem}

\begin{theorem}\label{teo:main2}
Problem \eqref{value_main}
has a solution,
and 
%, if $m:=\inf_{(h,\psi)\in X_l} \mathcal F_l(h,\psi)$, the infima in \eqref{plateau_intro1}, \eqref{eq:plateau_tilde_intro}, and on the left-hand side of \eqref{plateau_intro2} all coincide with $2m$, i.e.
\begin{align}
2\inf_{(h,\psi)\in X_l} \mathcal F_l(h,\psi)&=	\inf\{ \mathcal A(\psi,A_\sigma):(\sigma,\psi)\in S_{2l}\times 
\mathcal X_{D,\varphi},\;\psi=0\text{ on }\sigma((0,1))\}\nonumber\\
&=\inf\{ \mathcal A(\psi,SG_h):(h,\psi)\in \widetilde {\mathcal H}_{2l}\times 
\mathcal X_{D,\varphi},\;\psi=0\text{ on }G_h\}\nonumber\\
&=\inf_{(h,\psi) \in 
	X_{2\longR}^{\rm conv}}	\FB(h,\psi).
\end{align}
Furthermore, there is a constant $l_0>0$ such that
the following holds:
\begin{itemize}
\item[(i)]
 for $l \in (0,l_0)$ there is a minimizer $(h^\star,\psi^\star)\in  X_{2\longR}^{\rm conv}$ of $\FB$ satisfying conditions (1) and (2i)-(2iv) of Theorem \ref{teo:main1}, 
which is also a minimizer of 
\eqref{eq:plateau_tilde_intro}, $(h^\star\res [0,l],\psi^\star\res R_l)$ is a minimizer of $\mathcal F_l$, and the pair
$(\sigma^*, \psi^*)$, with
 $\sigma^\star(t):=(2lt,h^\star(t))$,
is a solution of \eqref{plateau_intro1};
\item[(ii)]
 For $l\geq l_0$ a common solution is $h^\star\equiv-1$ and $\psi\equiv0$.
\end{itemize}
\end{theorem}

The organization of the paper is as follows. Section 
\ref{sec:notation_and_preliminaries} contains some notation used throughout the paper.
Section \ref{sec:results} contains
some preliminary results on the functional $\FBl$ and its
doubled, 
on their domains and on relaxation.
Section \ref{sec:proof_of_theorems} contains the proofs of Theorems \ref{teo:main1} and \ref{teo:main2};
the most difficult part is contained in 
the regularity Theorem \ref{teo:boundary_regularity}.

%%%%%%%%%%%%%%%%%%%%%%%%%%%%%%%%%%%%%%%%%%%%%%%%%%%%%%%%%%%%%%%%%%%%%%%%%
\section{Notation and preliminaries}\label{sec:notation_and_preliminaries}
%%%%%%%%%%%%%%%%%%%%%%%%%%%%%%%%%%%%%%%%%%%%%%%%%%%%%%%%%%%%%%%%%%%%%%%%%
Let $\openset\subset\R^n$ be an open set, and let $v\in L^1(\openset)$.
We denote by $R(v)$ the set of regular points of $v$, 
{\it i.e.}, the set consisting of all $x\in \openset$ 
which are Lebesgue points for 
$v$ and $v(x)$ 
coincides with the Lebesgue value of $v$ at $x$. We  denote by $\overline v$ a good representative of $v\in L^1(\openset)$, i.e. a function such that $\overline v(x)$ coincides with the Lebesgue value of $v$ at $x$, for all $x\in R(v)$. If in addition $v$ is a function of bounded variation, we denote by $R_v$ the set of regular points $x\in R(v)$ such that $v$ is approximately differentiable at $x$. Notice that if $v\in BV(\openset)$ the set $R(v)\setminus R_v$ is $\mathcal L^n$-negligible.

%{\color{red} is approximately differentiable at $x$}{\color{blue} NON mi è chiaro cosa succede a questo insieme, se chiediamo la differenziabilità, quando $h$ non è Sobolev (ci serve all'inizio della sessione successiva)}.
 For a function $v\in L^1(\openset)\setminus BV(\openset)$ we also set
\begin{equation}\label{eq_GSG'}
 G'_v:=\{(x,\overline v(x))\in R(v)\times \R\}, \qquad
 SG'_v:=\{(x,y)\in R(v)\times \R:y< \overline v(x)\}.
 \end{equation}
If in addition $v\in BV(\openset)$, the previous sets are classically defined as above with $R(v)$ replaced by $R_v$, namely
 \begin{equation*}\label{eq_GSG}
 	G_v:=\{(x,v(x))\in R_v\times \R\}, \qquad
 	SG_v:=\{(x,y)\in R_v\times \R:y< v(x)\}.
 \end{equation*}
 Given a $2$-dimensional rectifiable set $S\subset U \subset \R^3$, $U$ open,
 and a tangent 
unit simple $2$-vector $\tau$ to it, we denote by 
$\jump{S}$ the current given by integration over $S$, namely
$$
\jump{S}(\openset)=\int_S\textlangle\tau(x),\omega(x)\textrangle~ 
d\mathcal H^{2}(x),
$$ 
$\omega$ a smooth $2$-form with compact support in $U$, 
see \cite{Krantz_Parks:08}, \cite{GiMoSu:98}, \cite{Federer:69}.
 We often will identify $SG'_v$ with the integral $3$-current $\jump{SG'_v}\in \mathcal D_3(\openset\times\R)$. If $v$ is a 
function of bounded variation, $\openset\setminus R_v$ has 
zero Lebesgue measure, so that the current $ \jump{SG_v}$ 
coincides with the standard integration over the subgraph of $v$.
It is well-known that the  perimeter of $SG_v$ in $\openset\times \R$ 
coincides with $\relarea(v,\openset)$.

The support of the boundary of $\jump{SG_v}$ 
includes the graph $G_v$, but in general consists also of additional parts, called vertical. We denote by
$$\mathcal G_v:=\partial \jump{SG_v}\res(\openset\times\R),$$
the generalized graph of $v$, which is a $2$-integral current supported on $\partial^*SG_v$, the reduced boundary of $SG_v$ in $\openset\times\R$. 

Let $\widehat \openset \subset \R^2$ 
be a bounded  open set such that $\openset\subseteq \widehat\openset$, 
and suppose that 
$L := \widehat \openset\cap
\partial \openset$ is a rectifiable 
curve. Given $\scalarfunction 
\in BV(\openset)$ and  a $W^{1,1}$ function $\varphi:
\widehat \openset \to \R$, we can consider 
\begin{align*}
 \overline{\scalarfunction}:=\begin{cases}
               f & \text{on }\openset,
\\
               \varphi&\text{on }\widehat \openset\setminus\openset.
               \end{cases}
\end{align*}
Then (see \cite{Giusti:84}, \cite{AmFuPa:00})
\begin{align*}
 \relarea(\overline{\scalarfunction},\widehat\openset)=
\relarea(\scalarfunction,\openset)+\int_{L}|\scalarfunction-\varphi|d\mathcal H^{1}
+\relarea(\varphi,\widehat\openset
\setminus\overline{\openset}).
\end{align*}

\subsection{Plateau problem in parametric form}\label{subsec:Plateau_problem_in_parametric_form}

We report here some 
results about the classical solution to the 
disc-type Plateau problem. Let $\unitdisc$
denote the open disc of radius $1$ centered at the origin of $\R^2$. 
If $\Gamma\subset\R^3$ is a closed rectifiable Jordan curve, the Plateau problem consists into minimize the functional
\begin{align}
\label{plateau}
\mathcal P_\Gamma(X):= \int_{\unitdisc}|\partial_{x_1}X\wedge\partial_{x_2} X|dx_1dx_2,
\end{align}
on the class of all functions $X\in C^0(\overline{B}_1;\R^3)\cap H^1(\unitdisc;\R^3)$ with
 $X\res\partial \unitdisc$ being a weakly monotonic parametrization of the curve $\Gamma$. 
The functional \eqref{plateau} measures the area (with multiplicity) of the surface $X(\unitdisc)$. 

A solution $X_\Gamma$ to the Plateau problem exists and satisfies the properties: 
it is harmonic (hence analytic)
\begin{align*}
 \Delta X_\Gamma=0\qquad\text{ in }D,
\end{align*}
it is a conformal parametrization 
\begin{align*}
 |\partial_{x_1}X_\Gamma|^2=|\partial_{x_2}X_\Gamma|^2,
\qquad \partial_{x_1}X_\Gamma\cdot \partial_{x_2}
X_\Gamma=0\qquad\text{ in }\unitdisc,
\end{align*}
and $X_\Gamma\res\partial \unitdisc$ is a strictly monotonic parametrization of $\Gamma$.
We will say that the surface $X_\Gamma(\unitdisc)$ has the topology of the disc.

Thanks to the properties above it is always possible, 
with the aid of a conformal change of variables, to parametrize $X(\unitdisc)$ over any simply connected bounded domain. In other words, if $U$ is any such domain, and if $\Phi:U\rightarrow \unitdisc$ is any conformal homeomorphism, then $X\circ\Phi$ is a solution to the Plateau problem on $U$.

%%%%%%%%%%%%%%%%%%%%%%%%%%%%%%%%%%%%%%%%%%%%%%%%%%%%%%%%%%%%%%%%%%%%%%%%%
\subsection{A Plateau problem for a self-intersecting boundary space curve}
\label{subsec:a_Plateau_problem_for_a_self-intersecting_boundary_space_curve}
%%%%%%%%%%%%%%%%%%%%%%%%%%%%%%%%%%%%%%%%%%%%%%%%%%%%%%%%%%%%%%%%%%%%%%%%%

The classical disc-type 
Plateau problem is solved for boundary value a simple Jordan space curve, 
in particular $\Gamma$ 
does not have self-intersections. Here we will treat a specific Plateau problem where the curve $\Gamma$ has non-trivial intersections, and it overlaps itself on a segment which is parametrized two times with opposite directions.

Specifically, we consider the cylinder $(0,2l)\times \unitdisc$ 
and 
two circles $\mathcal C_1, \mathcal C_2$
 which are the boundaries of its two circular bases, namely 
$\mathcal C_1:=\{0\}\times \partial \unitdisc$ and $\mathcal C_{2}:=\{2l\}\times \partial \unitdisc$. 
Then we take the segment $(0,2l)\times\{1\}\times\{0\}$. If $\gamma_0$ 
is a monotonic parametrization of this segment, 
starting from $(0,1,0)$ up to $(2l,1,0)$, $\gamma_1$ is a 
monotonic parametrization of $\mathcal C_1$ starting from the point $(0,1,0)$ and ending at the same point, and  $\gamma_2$ a parametrization of $\mathcal C_2$ with initial and final point $(2l,1,0)$
with the same orientation of $\mathcal C_1$, then we consider the parametrization 
\begin{align}\label{curvegamma_def}
 \gamma:=\gamma_1\star\gamma_0\star(-\gamma_2)\star(-\gamma_0),
\end{align}
(read from left to right)
which is a closed curve in $\R^3$ which 
travels two times across the segment $(0,2l)\times\{1\}\times\{0\}$ with opposite directions (the orientation of this curve is depicted in Fig. \ref{fig:curvagamma}).
We want to solve the Plateau problem with $\Gamma$ 
to be the image of $\gamma$.

\begin{figure}
\begin{center}
    \includegraphics[width=0.3\textwidth]{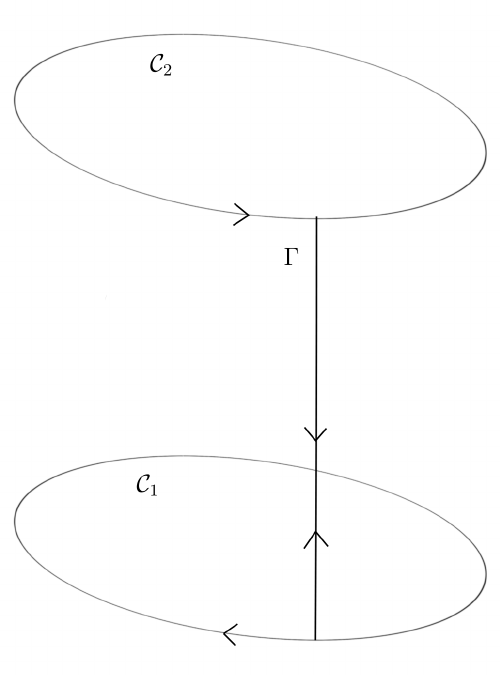}
 \caption{The self-overlapping curve $\Gamma$ with its orientation. }
\label{fig:curvagamma}
\end{center}
\end{figure}

The existence of solutions to the 
Plateau problem spanning self-intersecting boundaries has been addressed in \cite{Hass:91}, whose results have been recently improved in \cite{Creutz:20}. Without entering deeply into the details, it is known that, depending on the geometry of $\gamma$ (in this case, depending on the distance between the two circles $\mathcal C_1$ and $\mathcal C_2$) two kind of solutions are expected:
\begin{itemize}
 \item[(a)] The solution 
consists of two discs filling $\mathcal C_1$ and $\mathcal C_2$, see Fig. \ref{elliss_1}, right. In this case, a parametrization of it  $X:\overline \unitdisc
\rightarrow\R^3$ can be chosen so that, if $L_1$ and $L_2$ are two parallel chords in $\unitdisc$ dividing $\unitdisc$ in three sectors, then $X$ restricted to the sector enclosed between $L_1$ and $L_2$ parametrizes the segment $\gamma_0$ (and then its resulting area is null), $X(L_1)=P_1$ and $X(L_2)=P_2$ are the two endpoints of $\gamma_0$, and $X$ restricted to the sectors between $L_i$, $i=1,2$, and $\partial \unitdisc$ parametrizes the disc filling $\mathcal C_i$, $i=1,2$. Moreover the map $X$ can be still taken Sobolev regular 
(see \cite{Creutz:20} for details).

 \item[(b)] There is a classical solution, 
{\it i.e.}, there is a harmonic and conformal map $X:\unitdisc\rightarrow\R^3$,
continuous up to the boundary of $\unitdisc$, such that $X\res \partial \unitdisc$ is a weakly monotonic parametrization of $\Gamma$. 
In this case the resulting minimal surface is a sort of catenoid attached to the segment $(0,2l)\times \{(1,0)\}$ (see Fig. \ref{elliss_1} left).
\end{itemize}

\begin{remark}\label{rem:threshold}
We expect that there is a threshold $l_0$ such that if $l< l_0$ an area-minimizing disc with boundary $\gamma$ is of the form (b),
and for values $l>l_0$ the two discs have minimal area. We do not find explicitly $l_0$ but 
it is easy to see that if $l\leq \frac12$ an area-minimizing disc with boundary $\gamma$ has always less area than the solution with two discs. Indeed, the area of the two discs is $2\pi$, whereas we can always compare 
the area of the surface  $\Sigma$ as in (b) with the area  of the lateral surface of the cylinder $(0,2l)\times \unitdisc$, that is $4l\pi$. Hence $\mathcal H^2(\Sigma)<4l\pi\leq 2\pi$ for $l\leq \frac12$.	
\end{remark}

		\begin{figure}
		\includegraphics[width=0.8\textwidth]{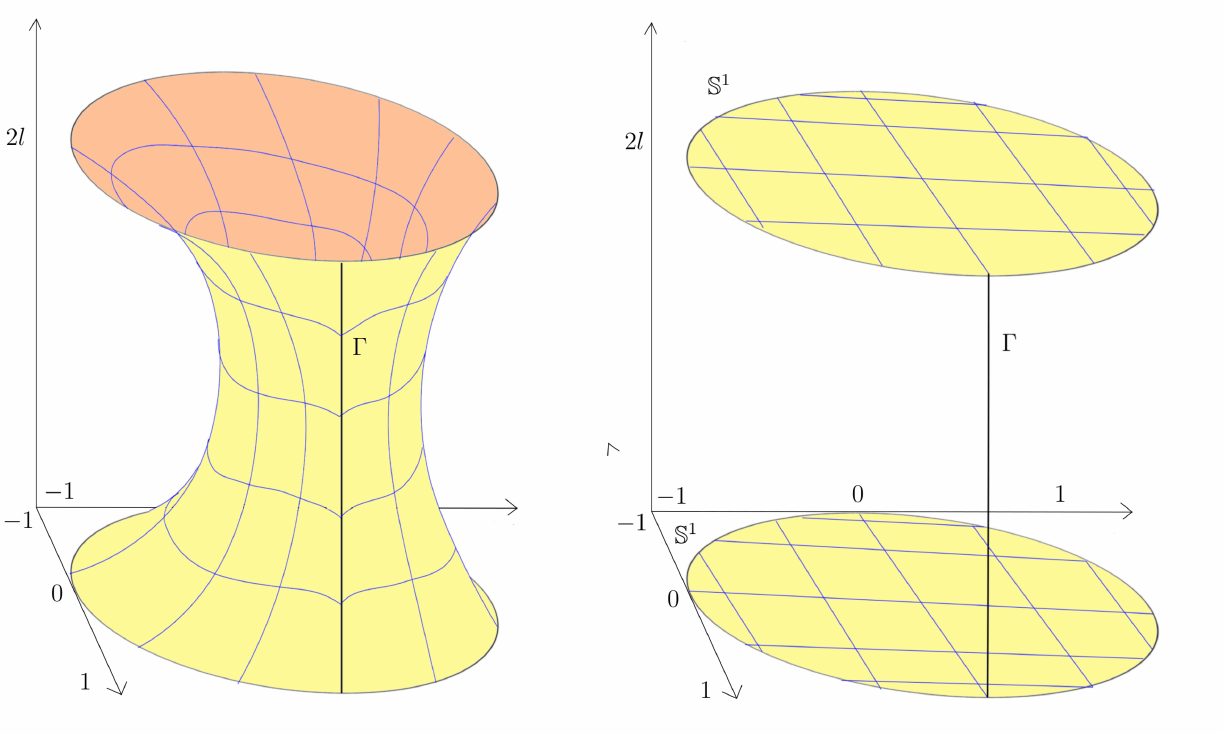}
		\caption{on the left the shape of a possible solution to the Plateau problem with boundary $\Gamma$. 
On the right another solution to the Plateau problem with boundary $\Gamma$.
See Section \ref{subsec:a_Plateau_problem_for_a_self-intersecting_boundary_space_curve}}\label{elliss_1}
		\end{figure}

%%%%%%%%%%%%%%%%%%%%%%%%%%%%%%%%%%%%%%%%%%%%%%%%%%%%%%%%%%%%%%%%%%%%%%%%
\section{Preliminary results on the functional $\FBl$ 
and its doubled}\label{sec:results}

For all $\varrho>0$ we denote 
$R_\varrho:=(0,\varrho)\times (-1,1)$.
For any $h\in L^\infty([0,\varrho];[-1,1])$,
we denote
by $G'_h$
 and $SG'_h$ the sets in \eqref{eq_GSG'}.
We recall that  $\varphi$ has been defined in \eqref{eq:varphi}.
\begin{definition}[\textbf{The functional $\FBl$}]
\label{def:the_functional_Fhat}
Let $l>0$ be fixed. Given 
$h\in L^\infty([0,\longR];[-1,1])$ and $\psi \in \BV(R_l;[0,1])$
we define 
\begin{equation}\label{eq:F'}
  \FBl(h,\psi):=\areaonecod
(\psi, R_l)-\mathcal H^2(R_l\setminus SG'_h)+\int_{\partial_DR_l}|\psi-\varphi|~d\mathcal H^1+\int_{(0,\longR)\times \{1\}}|\psi|~d\mathcal H^1.
\end{equation}
\end{definition}
We notice that if $h$ is also a function of bounded variation, then we can also write
\begin{equation}\label{eq:F}
	\FBl(h,\psi):=\areaonecod
	(\psi, R_l)-\mathcal H^2(R_l\setminus SG_h)+\int_{\partial_DR_l}|\psi-\varphi|~d\mathcal H^1+\int_{(0,\longR)\times \{1\}}|\psi|~d\mathcal H^1.
\end{equation}
We further remember from \eqref{eq:X_l_intro} the definition of $\oldDom$:
\begin{equation}\label{eq:X_l}
\oldDom
:=\{(h,\psi):
h\in L^\infty([0,\longR];[-1,1]), \psi \in \BV(R_l;[0,1]),\psi=0 ~{\rm in ~}R_l\setminus SG'_h
\}.
\end{equation}
In this section our concern is the analysis of the minimum problem
\begin{align}\label{min_Fbl}
\inf_{(h,\psi)\in \oldDom}\FBl(h,\psi).
\end{align}
Notice that in minimizing $\FBl$ we have a free boundary condition on the edge $\{l\}\times [-1,1]$.

\begin{remark}\label{leb_point}
 Let $(h,\psi) \in \oldDom$. 
If $t_0\in(0,\longR)$ is a regular point for $h$ (i.e. $t_0\in R(h)$) and if $\overline h(t_0)<1$, then the trace of $\psi$ over the segment $\{w_1=t_0,\;\overline h(t_0)\leq w_2\leq 1\}$ 
vanishes. Indeed
for any $\eta>0$ we can find $\delta_\eta>0$ such that  
 \begin{align}\label{markov1}
  \frac{1}{2\delta}\int_{t_0-\delta}^{t_0+\delta}|h(w_1)-\overline h(t_0)|~dw_1<\eta
\qquad 
\forall
\delta \in (0,\delta_\eta).
 \end{align}
Let now $s_0\in (-1,1)$ be such that $\overline h(t_0)<s_0\leq 1$ ({\it i.e.}, $(t_0,s_0)\in \{w_1=t_0,\;w_2> \overline h(t_0)\}$), and set
$2\Delta:=s_0-\overline h(t_0)$.
By 
Chebyschev inequality and \eqref{markov1} it follows that 
\begin{align}\label{markov2}
 \mathcal H^1(B_\Delta)\leq \frac{2\delta\eta}{\Delta}
\qquad {\rm where}~ \ \  
B_\Delta:=\{w_1\in(t_0-\delta,t_0+\delta):|h(w_1)-\overline h(t_0)|>\Delta\}.
\end{align}
Then, for any $\xi\in(0,\Delta)$ we 
infer\footnote{In the first inequality we have used that $0 \leq \psi\leq1$; 
in the second inequality 
that $SG'_h$ is the subgraph of $h$ in $(0,\longR)\times(-1,1)$;
 in the third inequality we have used that $s_0-h(t_0)=2\Delta$ and that $\xi<\Delta$. }
\begin{equation}
\label{markov3}
\begin{aligned}
 &\frac{1}{2\delta}\int_{t_0-\delta}^{t_0+\delta}
\int_{s_0-\xi}^{s_0+\xi}
\psi(w_1,w_2)~dw_2 dw_1 
\leq  \frac{1}{2\delta}\int_{t_0-\delta}^{t_0+\delta}
\int_{s_0-\xi}^{s_0+\xi}\chi^{}_{\{\psi>0\}}(w_1,w_2)~dw_2 dw_1
\\
\leq & \frac{1}{2\delta}\int_{t_0-\delta}^{t_0+\delta}
\int_{s_0-\xi}^{s_0+\xi}\chi^{}_{SG_h}(w_1,w_2)~dw_2dw_1
\leq 
\frac{\xi}{\delta}\int_{t_0-\delta}^{t_0+\delta}
\chi_{B_\Delta}(w_1)
dw_1\leq \frac{2\xi\eta}{\Delta},
\end{aligned}
\end{equation}
where the penultimate inequality follows from the inclusions 
\begin{align*}
SG'_h\cap\big(
[t_0-\delta,t_0+\delta]\times[s_0-\xi,s_0+\xi]\big)&\subseteq SG'_h\cap \big([t_0-\delta,t_0+\delta]\times[s_0-\Delta,s_0+\Delta]\big)\\
&\subseteq  B_\Delta\times[s_0-\Delta,s_0+\Delta],
\end{align*}
and the last inequality follows from \eqref{markov2}. 
Now \eqref{markov3} 
entails the claim by the arbitrariness of $\eta>0$ and since $\psi\geq0$.
\end{remark}

Now, we refine the choice of the class of pairs $(h,\psi)$ 
appearing in the infimum in \eqref{min_Fbl}.

\begin{definition}[\textbf{The classes $\Hricc$ and 
 $\Wspace$}]\label{def:the_class_W}
We set 
\begin{equation*}
\begin{aligned}
&\Hricc:=\{h\in L^\infty([0,\longR];[-1,1]): \;h{\rm ~convex~ and~ 
nonincreasing~ in~ }[0,\longR],\; h(0)=1\}, 
\\
& \Wspace:=\{(h,\psi) \in X_l:
~ h \in \Hricc\}.
\end{aligned}
\end{equation*}
\end{definition}

\begin{prop}[\textbf{Convexifying $h$}]\label{modifications_of_h}
We have
\begin{align}\label{get_rid_of_eps}
 \inf_{(h,\psi)\in \oldDom} 
\FBl(h,\psi)=\inf_{(h,\psi)\in \Wspace}\FBl(h,\psi).
\end{align}
\end{prop}
\begin{proof}
It is enough to show
the inequality ``$\geq$''.
By extending $\psi$ outside the open rectangle $ R_l$ as
$\psi:=0$ in $((0,\longR)\times \R)\setminus R_l$,  
we see that 
\begin{equation}\label{eq:by_extending}
\FBl(h,\psi)=
\areaonecod
\left(\psi, \overline R_l\setminus \big(\{w_1=0\}\cup\{w_1=l\}\big)\right)
-\mathcal H^2(R_l\setminus SG'_h)
+
\int_{\{0\}\times[-1,1]}|\psi^--\varphi|~d\mathcal H^1,
\end{equation}
where (see \eqref{eq:relaxed_area_one_cod})
$$\areaonecod\left(\zeta, \overline R_l\setminus \big(\{w_1=0\}\cup\{w_1=l\}\big)\right)
=\areaonecod
(\zeta,  {R}_l)+\int_{(0,\longR)\times \{1,-1\}}|\zeta^-|~d\mathcal H^1,$$
$\zeta^-$ being the trace of $\zeta \in BV(R_l)$ on $(0,\longR) \times\{1,-1\}$.

The thesis of the proposition will follow from the next three observations:
\begin{itemize}
 \item[(1)] If $h\in \Hricc$ 
is such that  $\overline h(t_0)=-1$ for some regular  point $t_0\in(0,\longR)$, 
then the subgraph $SG_h$ 
of $h$
splits in two mutually disjoint components: 
$(SG'_h)^-=SG'_h\cap \{w_1<t_0\}$ 
and $(SG'_h)^+=SG'_h\cap \{w_1>t_0\}$. 
Let  $\psi\in BV(R_l,[0,1])$ be such that 
\begin{align*}
 \psi=0 \qquad \text{~a.e.~in~ } R_l \setminus
SG'_h.
\end{align*}
The trace of $\psi$ over the segment $\{w_1=t_0,\;\overline h(t_0)\leq w_2\leq 1\}$ is $0$, as  a consequence of Remark \ref{leb_point}.
Then the function $\modpsi: R_l \to [0,1]$ defined as
\begin{align*}
 \modpsi(w_1,w_2):=\begin{cases}
                            \psi(w_1,w_2)&\text{if }w_1<t_0,\\
                            0&\text{otherwise,}
                           \end{cases}
\end{align*}
 still satisfies $(h,\modpsi)\in 
\oldDom$, and
\begin{align*}
\FBl(h,\modpsi)\leq \FBl(h,\psi).
\end{align*}
Being $\modpsi$ identically zero in $\{w_1> t_0\}$, in particular 
in $SG'_h\cap\{w_1>t_0\}$, we can introduce
\begin{align*}
 \modh(w_1):=\begin{cases}
                             h(w_1)&\text{if }w_1<t_0,\\
                            -1&\text{otherwise,}                       
                      \end{cases}
\end{align*}
so that $(\modh,\modpsi)\in \oldDom$
and we easily see that
$\FBl(\modh,\modpsi) 
 \leq \FBl(h,\modpsi)$;  
hence
\begin{equation*}
 \FBl(\modh,\modpsi) \leq \FBl(h,\psi).
\end{equation*}

\item[(2)] More generally, let $( h,\psi)\in
\oldDom$ and let $t_0 \in (0,\longR)$ 
be any regular point of $h$; we can also suppose that $\overline h(t_0)<1$. 
Consider 
\begin{align}\label{modif_omega1}
 \modh(w_1):=\begin{cases}
                             h(w_1)&\text{if }w_1<t_0,\\
                             h(w_1)\wedge \overline h(t_0)&\text{otherwise,}                       
                      \end{cases}
\end{align}
\begin{align*}
 \modpsi(w_1,w_2):=\begin{cases}
                            \psi(w_1,w_2)&\text{if }w_1<t_0,\\
                            \psi(w_1,w_2)&\text{if }w_1\geq t_0,\;w_2\leq \overline h(t_0),\\
                            0&\text{otherwise. }
                           \end{cases}
\end{align*}
We
claim that $\FBl(\modh,\modpsi) \leq \FBl( h,\psi).$
Define \begin{equation*}
  U:=\{(w_1,w_2) \in (0,\longR)\times(-1,1): w_1>t_0,\; \overline h(t_0)<w_2< h(w_1)\},
\end{equation*}
that is the set where we have replaced $\psi$ by $0$. 
To prove the claim, using \eqref{eq:by_extending} and
the equalities
$$
\int_{\{0\}\times[-1,1]}|\psi^--\varphi|~d\mathcal H^1 = 
\int_{\{0\}\times[-1,1]}|{\modpsi}^--\varphi|~d\mathcal H^1,
$$
$$
\mathcal H^2
(R_l \setminus SG'_{\modh}) = 
\mathcal H^2
(U \cup (R_l \setminus SG'_{h}))=
\mathcal H^2
(U) +
\mathcal H^2(R_l \setminus SG'_{h}),
$$
 we have to show that
\begin{align}\label{conclusion1}
\areaonecod
\left(\modpsi,
\overline R_l\setminus \big(\{w_1=0\}\cup\{w_1=l\}\big)
\right)\leq \areaonecod
\left(\psi,
\overline R_l\setminus \big(\{w_1=0\}\cup\{w_1=l\}\big)\right)
+\mathcal H^2( U).
\end{align}
Assume that $U$ is non-empty and that $\mathcal H^2(U)>0$.
It is convenient to introduce
$$V:=\{(w_1,w_2) \in R_l:t_0 < w_1<l,\; h(w_1)\vee \overline h(t_0)\leq w_2< 1\},$$
so that 
$U\cup V=\{(w_1,w_2): w_1>t_0,\; \overline h(t_0)<w_2< 1\}$ is an open rectangle. 
Since we have modified $\psi$ only in $U$, inequality 
\eqref{conclusion1} is equivalent to 
\begin{equation}
\label{conclusion1ter}
\begin{aligned}
& 
\areaonecod
(\modpsi,{U\cup V})+\int_{(t_0,l)\times\{h(t_0)\}}|{\modpsi}^+-{\modpsi}^-|d\mathcal H^1+\int_{(t_0,l)\times\{1\}}|{\modpsi}^-|d\mathcal H^1
\\
\leq
& 
\areaonecod
(\psi,{U\cup V})+\int_{(t_0,l)\times\{h(t_0)\}}|\psi^+-\psi^-|d\mathcal H^1+\int_{(t_0,l)\times\{1\}}|\psi^-|d\mathcal H^1
+\mathcal H^2( U),
\end{aligned}
\end{equation}
with 
$\psi^\pm$ (resp.
${\modpsi}^\pm$)
the external and internal traces of $\psi$ 
(resp. $\modpsi$) on $\partial (U\cup V)$;
here we have used from Remark \ref{leb_point}
that the trace of $\psi$ on $\{t_0\}\times (h(t_0),1)$ 
is zero (hence $
\int_{\{t_0\} \times (\overline h(t_0),1)} \vert \psi^+ - \psi^-\vert 
d\mathcal H^1=
\int_{\{t_0\}\times (\overline h(t_0),1)} 
\vert {\modpsi}^+ - {\modpsi}^-\vert d\mathcal H^1=
0$)
 and that the external traces $\psi^+$, ${\modpsi}^+$ 
on $(t_0,l)\times\{1\}$ vanish as well. 
Hence, exploiting that $\modpsi=0$ on $U\cup V$, so that 
$\areaonecod(\modpsi, U \cup V) = 
\mathcal H^2(U)+ 
\mathcal H^2(V)$, and that $\modpsi = \psi$ on $R_l \setminus (U\cup V)$,
inequality \eqref{conclusion1ter} is equivalent to
\begin{equation}\label{eq:conclusion2}
\begin{aligned}
&\mathcal H^2(V)+\int_{(t_0,l)\times\{\overline h(t_0)\}}|\psi^+|d\mathcal H^1
\\
\leq &\areaonecod
(\psi,{U\cup V})+\int_{(t_0,l)\times\{\overline h(t_0)\}}|\psi^+-\psi^-|d\mathcal H^1+\int_{(t_0,l)\times\{1\}}|\psi^-|d\mathcal H^1.
\end{aligned}
\end{equation}
We split
$$(t_0,l)=H_1\cup H_2\cup H_3,$$
with $H_1:=\{w_1\in (t_0,l):h(w_1)=1\}$, 
$H_2:=\{w_1\in (t_0,l):\overline h(t_0)\leq h(w_1)<1\}$, and 
$H_3:=\{w_1\in (t_0,l):h(w_1)< \overline h(t_0)\}$.
Since $\overline{\mathcal A}(\psi;U\cup V)=\mathcal H^2(\mathcal G_\psi\cap ((U\cup V)\times\R))$, by slicing and looking at $\mathcal G_\psi$ as an 
integral current 
\cite{Krantz_Parks:08}, \cite{GiMoSu:98}, \cite{Federer:69}, 
we have\footnote{Here we use that $D_{w_2}\psi=0$ in $V$.}
\begin{align*}
 \areaonecod
(\psi, U\cup V)&\geq \int_{(t_0,l)}
\mathcal H^1\Big((\mathcal G_{\psi})_{t} {\cap 
((t_0,l)\times (\overline h(t_0),1)\times \R})\Big)~dt
\\
 &\geq
\int_{(t_0,l)}\int_{(\overline h(t_0),1)}|D_{w_2}\psi(t,s)|~\;dt
+\mathcal H^2(V) 
\\
 &=
\int_{H_1\cup H_2}\int_{(\overline h(t_0),1)}|D_{w_2}\psi(t,s)|~\;dt
+\mathcal H^2(V)
\\
 &\geq
 \int_{H_2}|\psi^-(t,\overline h(t_0))|~dt+
\int_{H_1}|\psi^-(t,\overline h(t_0))-\psi^-(t,1)|~d t
+\mathcal H^2(V)
\\
&\geq\int_{H_1\cup H_2}|\psi^-(t,\overline h(t_0))|~d t 
-\int_{H_1}|\psi^-(t,1)|~d t+\mathcal H^2(V)
\\
 &=\int_{(t_0,l)}|\psi^-(t,\overline h(t_0))|~d t 
-\int_{H_1}|\psi^-(t,1)|~d t+\mathcal H^2(V)
\\
 &=\int_{(t_0,l)}|\psi^-(t,\overline h(t_0))|~dt 
-\int_{(t_0,l)}|\psi^-(t,1)|~dt+\mathcal H^2(V),
\end{align*}
where $(\mathcal  G_{\psi})_{t}$ is the slice of $\mathcal G_\psi$ 
on the plane $\{w_1=t\}$, 
that is the generalized graph of the function $\psi\res\{w_2=t\}$. 
From the above expression, the 
triangular inequality 
implies \eqref{eq:conclusion2}.

\item[(3)] Let $(h,\psi)\in
\oldDom$. Let $t_1,t_2\in (\eps,l)$ 
be regular points for $h$ with  $t_1<t_2$, 
and let $r_{12}(t):=\overline h(t_1)+\frac{\overline h(t_2)-\overline h(t_1)}{t_2-t_1}(t-t_1)$. We consider the following modifications of $h$ and $\psi$:
\begin{align*}
 \modhbis(w_1):=\begin{cases}
                            h(w_1)&\text{if }0<w_1<t_1\text{ or }l>w_1>t_2,\\
                            h(w_1)\wedge r_{12}(w_1)&\text{otherwise,}                       
                      \end{cases}
\end{align*}
and 
\begin{align*}
 \modpsibis(w_1,w_2):=\begin{cases}
                            \psi(w_1,w_2)&\text{if }0<w_1<t_1
\text{ or }l>w_1>t_2,\\
                            \psi(w_1,w_2)&\text{if }w_1\in[t_1,t_2] ~\textrm{and} ~w_2\leq r_{12}(w_1),\\
                            0&\text{otherwise. }
                           \end{cases}
\end{align*}
In other words we set $\psi$ equal to $0$ above the segment $L_{12}$ connecting $(t_1,\overline h(t_1))$ to $(t_2,\overline h(t_2))$.
Also in this case we have
\begin{align}\label{ineq_conv}
 \FBl(\modhbis,\modpsibis) \leq \FBl(h,\psi).
\end{align}
Indeed, if $\overline h(t_1)=\overline h(t_2)$ the proof is 
identical to the case (2). Otherwise, it can be obtained by slicing as well, 
parametrizing $L_{12}$ by an arc length parameter, 
then slicing the region $\{(w_1,w_2):w_1\in  (t_1,t_2),\;w_2\in (\ell_{12}(w_1),1)\} $\footnote{$\ell_{12}$ represents the affine function whose graph is $L_{12}$.} by lines perpendicular to $L_{12}$, and 
exploiting the fact that $\psi$ equals zero on the segments $\{t_i\}\times (h(t_i),1)$.
\end{itemize}
 
Let  $(h,\psi)\in \oldDom$ be given; 
from (3) we
can always replace $h$ by its convex envelope 
and modifying accordingly 
$\psi$, we get two functions $\modhbis$ and $\modpsibis$ 
such that \eqref{ineq_conv} holds. Moreover, by (2), 
if $t_0\in(0,\longR)$ is a regular point for $h^\#$, we 
can always replace $h^\#$ by $\modh$ in \eqref{modif_omega1}, 
so that $\modh$ turns out to be nonincreasing. 
The assertion of the proposition follows.
\end{proof}

Let us rewrite the functional $\FBl$ in a convenient way. 
Let  $(h,\psi)\in \Wspace$, and let $G_h=
\{(w_1,h(w_1)):w_1\in (0,\longR)\}\subset \overline{R}_l$ 
be the graph of $h$. We have, using \eqref{eq:F},
\begin{equation}\label{eq:F_ter}
  \FBl(h,\psi)
=
  \areaonecod
(\psi, SG_h)+\int_{\graphh \setminus \{h=-1\}}|\psi|~d\mathcal H^1+\int_{\partial_DR_l}|\psi-\varphi|~d\mathcal H^1,
\end{equation}
where, in the integral over $\graphh$, 
we consider the trace of $\psi \res SG_h$ on $\graphh$. 

%%%%%%%%%%%%%%%%%%%%%%%%%%%%%%%%%%%%%%%%%%%%%%%%%%%%%%%%%%%%%%%%%%%%%%%%%%%%%%
%%%%%%%%%%%%%%%%%%%%%%%%%%%%%%%%%%%%%%%
\subsection{Doubling}\label{subsec:doubling}
Now we analyse the minimum problem on the
right-hand side of \eqref{get_rid_of_eps}. To this aim,
as explained in the introduction,  it is convenient to write the analogue of $\FBl$ 
in a doubled rectangle, see \eqref{eq:F_2l}.
 
Remember  that $\doubledrectangle$ 
denotes the open doubled 
rectangle, $\doubledrectangle
:=(0,2l)\times (-1,1)$; we define its Dirichlet boundary
\footnote{It is worth noticing once more 
that $\partial_D \doubledrectangle$ consists of three edges of 
$\partial \doubledrectangle$,
while $\partial_D R_l$ (see \eqref{eq:Dirichlet_part_of_boundary}) 
consists of two edges of $\partial R_l$.} $\partial_D 
\doubledrectangle\subset \partial \doubledrectangle$ as
\begin{align*}
 \partial_D \doubledrectangle:=\big(\{0,2l\}\times [-1,1]\big)\cup \big((0,2l)\times\{-1\}\big),
\end{align*}
so that 
$\partial \doubledrectangle \setminus \partialbar\doubledrectangle = (0,2l)\times \{1\}$.	

We recall that $\Hspace$ has been defined in 
\eqref{eq:space_of_h_in_the_doubled_interval}, and that
for each $h\in \Hspace$,
 $$
\graphh :=\{(\axialcoordofcylinder, h(\axialcoordofcylinder)) 
: \axialcoordofcylinder \in (0,2\longR)\}, \qquad
\subgraphh 
:= \{(\axialcoordofcylinder,\scoord)\in \doubledrectangle: \scoord<\h(\axialcoordofcylinder)\},
$$
where $\subgraphh:= \emptyset$ 
in the case $h\equiv-1$. Notice that for $h\in \Hspace$, $SG_h$ is an open set.
We set
\begin{equation}\label{eq:Lh}
\Lh:=
\Big(\{0\} \times (h(0),1)\Big) \cup 
\Big(\{2l\} \times (h(2l),1)\Big),
\end{equation}
which is either empty,
or the union of two equal intervals,
see Fig. \ref{fig:graphofh}.

\begin{figure}
\centering
\includegraphics[scale=0.5]{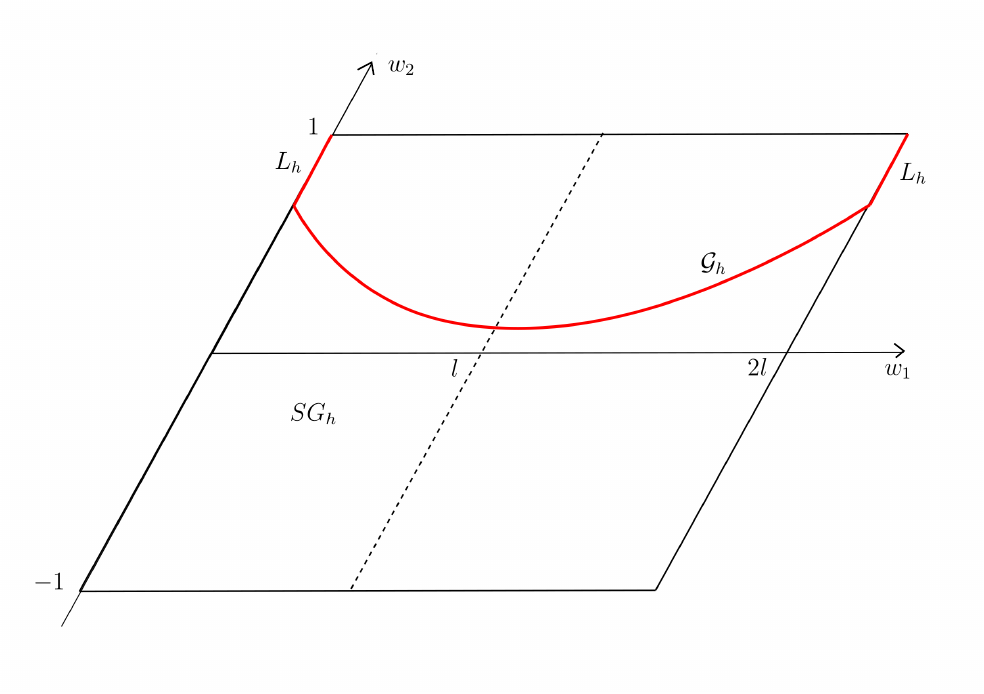}
\caption{the graph of a convex symmetric 
function $\h \in \Hspace$;
$\Lh$, defined in \eqref{eq:Lh},  
consists of  the two vertical segments  
over the boundary of $(0,2\longR)$, from $h(0) = h(2\longR)$ to $1$.}
\label{fig:graphofh}
\end{figure}

Clearly, the restriction of $\varphi$, defined in \eqref{eq:varphi},
 on $\partial_D \doubledrectangle$ reads as:
\begin{equation}\label{eq:varphi_doubled}
\dirdatum(\axialcoordofcylinder,\scoord):=
\begin{cases}
\sqrt{1-\scoord^2} & \text{ if }
(\axialcoordofcylinder,\scoord)
 \in \{0,2\longR\} \times [-1,1],
\\
0 & \text{ if }(\axialcoordofcylinder,\scoord) 
\in (0,2\longR) \times \{-1\}.
\end{cases}
\end{equation}
The graph of $\dirdatum$ on $\{0,2\longR\} \times [-1,1]$ 
consists of two half-circles of radius $1$ 
centered at $(0,0)$ and $(2 \longR,0)$ respectively, 
see Fig. \ref{fig:graphofphi}. 

\begin{figure}
\centering
\includegraphics[scale=0.50]{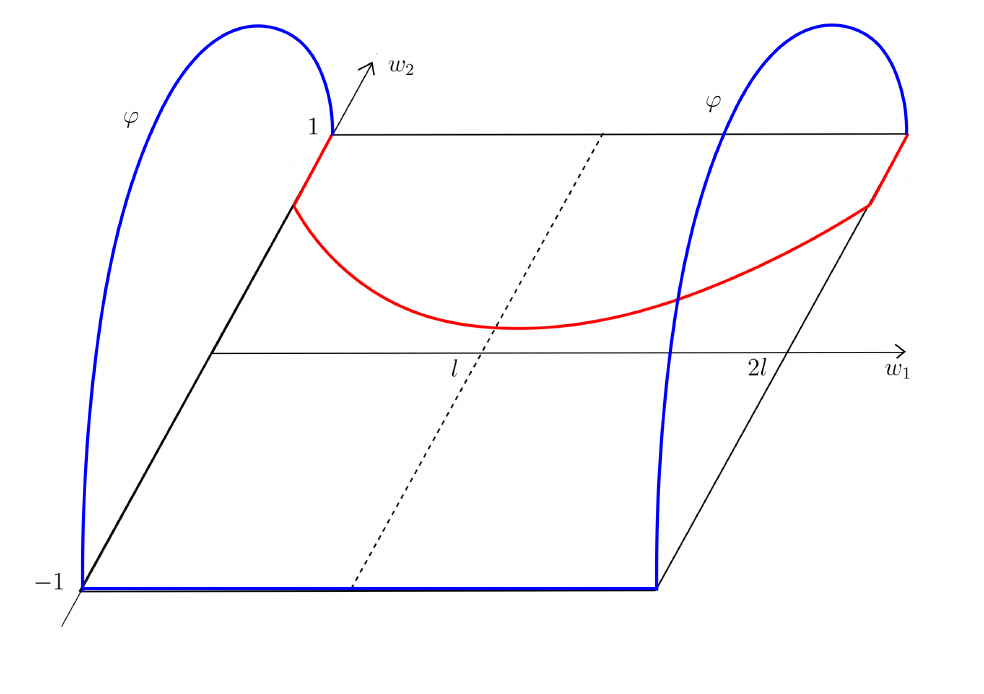}
\caption{the graph of the boundary condition function $\dirdatum$
in \eqref{eq:varphi_doubled} on the Dirichlet
boundary of $R_{2\longR}$. We also draw the graph of a function 
$h \in \mathcal H_{2\longR}$, and the
two segments $L_h$.
}
\label{fig:graphofphi}
\end{figure}

We further recall that
$X_{2\longR}^{\rm conv}$ has been defined in \eqref{eq:X_2l_conv} and that,
and for any $(h,\psi) \in 
X_{2\longR}^{\rm conv}$, 
$\FB(\h,\psione)$ has been defined in \eqref{eq:F_2l}.
\begin{Remark}\rm \label{rmk:FAremarks}
\begin{itemize}
\item[(i)] The only case in which
the last addendum on the right-hand side of \eqref{eq:F_2l} may be positive
is when $h$ is identically $1$ on $\partial \doubledrectangle \setminus 
\partial_D \doubledrectangle$;
 \item[(ii)] 
We claim that 
\begin{equation}
\FB(\h,\psione)=\Aone(\psione,\Omegah)
+
\int_{\partialbar \Omegah}  \vert \psione  -\dirdatum\vert ~d \Hone 
+ 
\int_{\graphh \setminus\{w_2=-1\}}\vert \psione^- \vert ~d \Hone 
+
\int_{\Lh} \dirdatum ~d \Hone, 
\label{eq:FA}
\end{equation}
where
\begin{equation}\label{eq:partial_D_G_h}
\partial_D \Omegah:= (\partial_D \doubledrectangle) \cap \partial \Omegah,
\end{equation}
and $\psi^-$ denotes the trace of $\psi$ from the side of $SG_h$. 

To show \eqref{eq:FA}, we start to observe that,  
 using that $\psione=0$ on $\doubledrectangle \setminus \Omegah$, it follows 
$\int_{\partial \doubledrectangle \setminus 
\partial_D \doubledrectangle} \vert \psi\vert~d \mathcal H^1 = 
\int_{G_h\cap\{w_2=1\}} \vert \psi\vert~d \mathcal H^1$. This last term is nonzero only if $h\equiv 1$, in which case $L_h$ is empty, and the equivalence between \eqref{eq:F_2l} and \eqref{eq:FA} easily follows. If instead $h$ is not identically $1$, then, using again that $\psione=0$ 
on $\doubledrectangle \setminus \Omegah$, we see that the last term on 
the right-hand side of \eqref{eq:F_2l} is null,  
and from \eqref{eq:relaxed_area_one_cod},
\begin{equation}\label{eq:AA}
\Aone(\psione,\Omegah)=
\Aone(\psione,\doubledrectangle)-\Htwo(\doubledrectangle \setminus \Omegah) -
\int_{\graphh \cap \doubledrectangle} \vert\psione^- \vert d \Hone;
\end{equation}
hence,
inserting
\eqref{eq:AA} into \eqref{eq:F_2l}, 
we obtain, splitting $\partial_D \doubledrectangle
  = (\partial_D \Omegah) \cup L_h \cup (G_h\cap \{w_2=-1\})$, and using
that $\dirdatum =0$ on $(0,2\longR)\times\{-1\}$,
\begin{equation*}
\begin{aligned}
\FB(\h,\psione)
=&
\Aone(\psione,\doubledrectangle)-\Htwo(\doubledrectangle \setminus \Omegah) 
+\int_{\partialbar \doubledrectangle}  \vert \psione -\dirdatum\vert d \Hone 
\\
=&\Aone(\psione,\Omegah)
+\int_{\partialbar \doubledrectangle}  \vert \psione -\dirdatum\vert d \Hone
+\int_{\graphh \cap \doubledrectangle} \vert\psione^- \vert d \Hone
\\
=&
\Aone(\psione,\Omegah)+\int_{\partialbar \Omegah}  
\vert \psione -\dirdatum\vert d \Hone + 
\int_{\Lh} \vert  \dirdatum\vert d \Hone 
+ \int_{G_h\cap\{w_2=-1\}} \vert \dirdatum\vert d \Hone
\\
& 
+\int_{\graphh \cap \doubledrectangle} \vert\psione^- \vert d \Hone
\\
=& \Aone(\psione,\Omegah)+
\int_{\partialbar \Omegah}  \vert \psione  -\dirdatum\vert d \Hone 
+\int_{\graphh\setminus \{w_2=-1\}}\vert 
\psione^-\vert d\Hone + \int_{\Lh} \dirdatum d \Hone.
\end{aligned}
\end{equation*}
\item[(iii)] We have 
 \begin{equation}  \label{eq:equiv_of_inf_on_good_set}
 \begin{aligned}
& \inf_{(h,\psi)\in
 X_{2\longR}^{\rm conv}} \FB(\h,\psione)
\\
=&
\inf \big \{ \Aone(\psione,\Omegah)
+
\int_{\partialbar \Omegah}  \vert \psione  -\dirdatum\vert ~d \Hone 
+ 
\int_{\graphh \setminus \{w_2=-1\}}\vert \psione^-\vert ~d \Hone 
+
\int_{\Lh} \dirdatum ~d \Hone
\\
 & \quad \quad :\h \in \Hspace \setminus \{\h\equiv -1\},~\psione \in  
\BV (\Omegah,[0,1]) \big \} .
\end{aligned} 
\end{equation}
 \item[(iv)] If $h>-1$ everywhere, then  $\Omegah$ is connected,
$\partial_D \Omegah = \partial_D \doubledrectangle \setminus \Lh$,
 and the sum of the first three terms on 
the right-hand side of \eqref{eq:FA} gives the area 
of the graph of $\psi$ 
on $\overline{\Omegah}$, with the boundary condition $\dirdatum$ set to be $0$ on $\graphh$. 
 
\item[(v)]  Our aim is to have a surface in 
$\overline \doubledrectangle \times \R \subset \R^3=
\R^2_{(\axialcoordofcylinder,\scoord)}\times \R$ of graph type,  whose boundary consists of the union of the graph of $\dirdatum$ and 
the graph of a convex function $\h \in \Hspace$. The last three terms in \eqref{eq:FA} are an area penalization to force the solution to attain these boundary conditions by filling, with vertical walls, the gap between the boundary of any 
competitor surface (the generalized graph of $\psione$) and the required boundary conditions. In particular the presence of the last term of \eqref{eq:FA} is explained as follows: 
when $\h(0)<1$, \textit{i.e.}, $\Lh \neq \emptyset$, the graph of 
any $\psione \in \BV(\Omegah,[0,1])$ does not
reach the graph of $\dirdatum \vert_{\Lh}$ (simply because $L_h
\cap \overline \Omegah=\emptyset$). To overcome this, 
the graph of $\psione$ is glued to the wall
consisting of the subgraph of 
$\dirdatum \vert_{\Lh}$ (inside $\overline R_{2\longR}$).

\item[(vi)] Take $\hn:=-1+\frac{1}{n}$, and $\psionen:=c>0$ on $\Omegahn$,  
then $\displaystyle
\lim_{n\to +\infty} \areaonecod(\psi_n, \Omegahn)=0$, 
$\displaystyle
\lim_{n\to +\infty} \int_{\partialbar \Omegahn}  
\vert \psione_n  -\dirdatum\vert ~d \Hone 
=2cl$, and 
$\displaystyle
\lim_{n\rightarrow+\infty} \int_{G_{h_n} \setminus \{h_n=-1\}}\vert \psione \vert ~d \Hone =2cl$,
$\displaystyle
\lim_{n\to +\infty} \int_{L_{h_n}} \dirdatum~  d\Hone=\pi$,
hence
$$
\FB(-1,0)=\pi < \lim_{n\to +\infty} \FB(h_n,\psi_n) = 4c\longR +\pi,
$$
that is the functional $\FB$ in some sense forces a minimizing  sequence to attain the boundary conditions as much as possible. 
  \end{itemize}
 \end{Remark}

By symmetry, we easily infer
\begin{align}\label{doubling}
 2\inf_{(h,\psi)\in \Wspace}
\FBl(h,\psi)=\inf_{(h,\psi)\in
 X_{2\longR}^{\rm conv}}
\FB(h,\psi),
\end{align}
therefore, by \eqref{get_rid_of_eps}, 
\begin{align}\label{equivalenza}
	\inf_{(h,\psi)\in \oldDom} 
	\FBl(h,\psi)=\frac12\inf_{(h,\psi)\in
		X_{2\longR}^{\rm conv}}
	\FB(h,\psi).
\end{align}
Next we analyse the latter minimization problem.
\begin{Remark}[\textbf{Two explicit estimates from above}]
\label{rem:two_explicit_estimates_form_above}\rm 
Let $\h\equiv 1$ and 
$\psione(\axialcoordofcylinder,\scoord):=\sqrt{1-\scoord^2} = \varphi(w_1,w_2)$, for any $(w_1,w_2) \in R_{2\longR}$. 
Then $(\h, \psione)$ is one of the competitors  in \eqref{doubling} and 
therefore 
$$
\inf_{(h,\psi)\in
 X_{2\longR}^{\rm conv}} \FB(h,\psi)\leq 
\FB(1,\psione)=2 \pi \longR \qquad \forall \longR >0,
$$
which is the lateral area of the cylinder $(0,2\longR) \times \unitdisc$.
Also, $\FB$ is well-defined for $h\equiv -1$, 
in which case $\Omegah = \emptyset$, $\psi \equiv 0$ in $\doubledrectangle$,
and therefore
\begin{equation}\label{eq:FB_minus_one_zero} 
\FB(-1, 0)=\int_{\{0,2\longR\}\times(-1,1)}\dirdatum~ d\Hone =\pi,
\end{equation}
which is the area of  the two 
half-discs joined by the segment $(0,2 \longR) \times \{-1\} $,
see Fig. \ref{fig:graphofphi}.  In particular
\begin{equation} \label{eq:upperboundofA}
\inf_{(h,\psi)\in
 X_{2\longR}^{\rm conv}}
\mathcal{F}(h,\psi)\leq \pi \qquad 
\forall
 \longR >0.
\end{equation}
\end{Remark}

In the next section we shall 
prove the existence and regularity of minimizers for the minimum
problem on the right-hand side of \eqref{equivalenza}.

\section{Proof of Theorems \ref{teo:main1} and \ref{teo:main2}}\label{sec:proof_of_theorems}
This section is devoted to the proof of our main results,
Theorems \ref{teo:main1} and \ref{teo:main2},
and it is splitted into various subsections for clarity of the 
presentation.

\subsection{Existence of a minimizer of $\FB$ in $X_{2\longR}^{\rm conv}$}\label{subsec:existence_of_a_minimizer}
We start by analysing the features of the space 
$\Hspace$ defined
in \eqref{eq:X_2l_conv}.
Clearly the graph of $h \in \Hspace$ is symmetric with 
respect to $\{\axialcoordofcylinder=\longR\}$; also, the convexity of $\h$ implies $\h
\in {\rm Lip}_{{\rm loc}}((0,2\longR))$, 
and $\h$ has a  
continuous extension on $[0, 2\longR]$.

\begin{definition}[\textbf{Convergence in $X_{2\longR}^{\rm conv}$}]
\label{def:convergence_in_Dom_FB}
We say that a sequence $((\hn, \psionen))\subset 
 X_{2\longR}^{\rm conv}$  
converges to $(\h,\psione)\in  X_{2\longR}^{\rm conv}$, if 
\begin{itemize}
\item[-] $(\hn)$ converges 
to $\h$
uniformly on compact subsets of $(0,2\longR)$; 
\item[-] $(\psionen)$ converges 
to $\psione$ in $\Lone(\doubledrectangle)$.
\end{itemize}
\end{definition}

\begin{lemma}[\textbf{Compactness of $\Hspace$}]\label{lem:compactnessofH}
Every sequence $(\h_k) \subset \Hspace$ has a subsequence 
converging uniformly on compact subsets of $(0,2\longR)$ 
to some element of $\Hspace$.
\end{lemma}

\begin{proof} See for instance \cite[Sec. 1.1]{Hormander:94}.
\end{proof}

\begin{lemma}[\textbf{Closedness of $X_{2\longR}^{\rm conv}$}]\label{lem:closednessinB}
Let $((\hn,\psionen))\subset X_{2\longR}^{\rm conv}$ be a sequence such that 
$(\hn)$ converges 
to $\h \in \Hspace$
uniformly on compact subsets of $(0,2\longR)$,
 and $(\psionen)$ converges 
to $\psione \in \BV(\doubledrectangle)$
in $\Lone(\doubledrectangle)$.
Then $(\h, \psione)\in X_{2\longR}^{\rm conv}$.
\end{lemma}

\begin{proof}
Possibly passing to a (not relabelled) 
subsequence, we can assume that 
$(\psionen)$ 
converges to $\psione$ pointwise in $A \subseteq \doubledrectangle$, with 
$\Htwo (\doubledrectangle \setminus A)=0$, and
$\psionen =0$ in  $A \cap 
(\doubledrectangle \setminus \Omegahn)$
for all $n \in \mathbb N$.
We only have to show that
$\psione =0$ in $A \cap (\doubledrectangle \setminus \Omegah)$. We may assume
that $A$ does not intersect the graph of $h$. 
If $(\axialcoordofcylinder,\scoord) \in A \cap (\doubledrectangle \setminus \Omegah)$, 
then $\scoord>\h(\axialcoordofcylinder)$. {}From 
the local uniform convergence of $(\hn)$ to $\h$ 
in $(0,2\longR)$ it follows that
$\scoord>\hn(\axialcoordofcylinder)$  for $n$ large enough, {\it i.e.,} 
$(\axialcoordofcylinder,\scoord)\in A\cap (\doubledrectangle \setminus \Omegahn)$,
and the assertion follows.
\end{proof}

\begin{lemma}[\textbf{Lower semicontinuity of $\FB$}]
\label{lem:lowersemicontinuityofFB}
Let $((\hn, \psionen))\subset X_{2\longR}^{\rm conv}$ be a sequence 
converging  to $(\h,\psione) \in X_{2\longR}^{\rm conv}$ in the sense of Definition \ref{def:convergence_in_Dom_FB}. Then 
\begin{equation}\label{eq:partofF(B)}
\FB(\h,\psione) \leq \liminf_{n\to +\infty}\FB(\hn,\psionen).
\end{equation}
\end{lemma}

\begin{proof}

\smallskip
It is standard\footnote{Indeed, let $\tilde{\dirdatum}:\partial\doubledrectangle \to [0,1]$ be defined as 
$\tilde{\dirdatum}:=\dirdatum$  on $\partialbar \doubledrectangle$, and 
$\tilde\dirdatum := 0$ on $\partial \doubledrectangle \setminus \partialbar \doubledrectangle$.
Let $B\subset \R^2$ be an open disc containing $\overline \doubledrectangle$. 
We extend $\tilde\dirdatum $ to a $W^{1,1}$ function 
in $B \setminus \overline \doubledrectangle$,
 \cite[Thm. 2.16] {Giusti:84},
and we still denote by $\tilde\dirdatum$ such an extension.
For every $\psione \in BV(\doubledrectangle)$, define 
$\widehat \psione
:=\psione$ in $\doubledrectangle$ and
$\widehat \psione :=
\tilde\dirdatum$ in $B \setminus\doubledrectangle$.
We have 
$$
\begin{aligned}
 \Aone(\psione,\doubledrectangle)+\int_{\partialbar \doubledrectangle}  \vert \psione  -\dirdatum\vert d \Hone +\int_{\partial 
\doubledrectangle\setminus \partialbar \doubledrectangle}  \vert \psione \vert d \Hone
=  \Aone(\psione,\doubledrectangle)+\int_{\partial \doubledrectangle}  \vert \psione  -
\tilde \dirdatum\vert d \Hone 
=  \Aone(\widehat \psione,B) -\Aone(\tilde\dirdatum,B\setminus \overline \doubledrectangle),
\end{aligned}
$$
where the last equality follows from \cite[(2.15)] {Giusti:84}. 
Thus the lower semicontinuity of the functional in \eqref{eq:first_three_terms} 
follows from the $\Lone(B)$-lower semicontinuity of the area functional.}
 to show that
the functional 
\begin{equation}\label{eq:first_three_terms}
\psi \in BV(\doubledrectangle, [0,1]) \to \Aone(\psione,\doubledrectangle)+\int_{\partialbar \doubledrectangle}  
\vert \psione  -\dirdatum\vert~ d \Hone +\int_{\partial \doubledrectangle\setminus \partialbar \doubledrectangle} 
 \vert \psione \vert ~d \Hone
\end{equation}
is $\Lone(\doubledrectangle)$-lower semicontinuous.
Since $(\hn)$ converges to $\h$ pointwise in $(0,2\longR)$, we also have 
$\lim_{n \to +\infty} \Htwo (\doubledrectangle \setminus \Omegahn) = \Htwo (\doubledrectangle \setminus \Omegah)$.
The assertion follows.
\end{proof}

The existence statement of
Theorem \ref{teo:main1}
is given by the following

\begin{prop}[\textbf{Existence of a minimizer of \eqref{eq:B_intro}}]\label{Thm:existenceofB}
The minimum
problem \eqref{eq:B_intro}
has a solution.
\end{prop}

\begin{proof}
The pair $(h_0,\psi_0)$ given by 
$$ 
h_0(\axialcoordofcylinder):=-1,\quad 
\psi_0(\axialcoordofcylinder, \scoord):=0, \qquad (\axialcoordofcylinder,\scoord)\in \doubledrectangle,
$$
is a competitor in \eqref{eq:B_intro}. Hence, 
for a minimizing sequence $((\hn, \psionen))\subset X_{2\longR}^{\rm conv}$, 
recalling \eqref{eq:FB_minus_one_zero} 
we have
\begin{equation}
\lim_{n\rightarrow +\infty} \FB(\hn, \psionen)=\inf \big \{ \FB(\h,\psione): (\h,\psione) \in 
X_{2\longR}^{\rm conv}
 \big \} \leq \pi.
\end{equation}
Thus 
$\sup_{n \in \mathbb N}
\vert D \psionen \vert (\doubledrectangle) < +\infty$, and 
there exists $\psi \in BV(\doubledrectangle, [0,1])$ such that, up to a 
(not relabelled) subsequence, $(\psi_n)$ 
converges to $\psi$ in $\Lone(\Omega)$. 

Using Lemmas \ref{lem:compactnessofH} and 
\ref{lem:closednessinB}, we may assume that $(\h_n)$ 
converges locally uniformly to some $h \in \Hspace$,  
and $\psi=0$ in $\doubledrectangle \setminus \Omegah$.
The assertion then follows from 
Lemma \ref{lem:lowersemicontinuityofFB}.
\end{proof}

Now, we turn to the regularity and qualitative properties of minimizers.
\subsection{Regularity of minimizers of  $\FB$ in $X_{2\longR}^{\rm conv}$}

The next proposition shows (2ii) in Theorem \ref{teo:main1}.
\begin{prop}[\textbf{Analyticity and positivity of a minimizer}]
\label{prop:reg}
Suppose that $(\h,\psione)$ is a 
minimizer of \eqref{eq:B_intro}, and that $h$ is not identically $-1$.
Then $\psione$ is analytic 
in $\Omegah$, and 
\begin{align}\label{div_eq_bis}
{\rm div} \Big(\frac{\nabla \psione}{\sqrt{1+|\nabla \psione|^2}}\Big)=0
\qquad {\rm in}~ \Omegah.
\end{align}
Moreover 
\begin{equation}\label{eq:psi_is_positive}
\psione>0 \quad{\rm in} ~\Omegah.
\end{equation}
\end{prop}

\begin{proof}
Since by assumption $h$ is not identically $-1$,
we have that $\Omegah$ is nonempty. Moreover minimality ensures
$$
\int_{\Omegah} \sqrt{1+|D \psione|^2} \leq 
\int_{\Omegah} \sqrt{1+|D \psione_1|^2} 
 $$ 
for any $\psi_1 \in BV(\Omegah)$ with ${\rm spt}(\psi
-\psi_1) \subset \subset \Omegah$.
Thus, 
by  \cite[Thm
14.13]{Giusti:84}, $\psione$ is locally Lipschitz, and hence analytic, in
$\Omegah$, and
\eqref{div_eq_bis} follows. 
Now, let $z\in \Omegah$ and take an open  
disc $B_\eta(z)\subset \subset \Omegah$. Since
$\psione \geq 0$ on $\partial B_\eta(z)$ we find, by
the strong maximum principle  \cite[Thm. C.4]{Giusti:84}, that 
either $\psione$ is identically zero  in $B_\eta(z)$, or
$\psione>0$ in
$B_\eta(z)$. Hence from the analyticity of $\psione$ and the
arbitrariness of
$z$,  we have that either $\psione$ is identically
zero in $\Omegah$ or $\psione>0$ in $\Omegah$. 
Now 
$\FB(h,0)=|\Omegah|+\pi > \FB(-1, 0)=\pi$, see
\eqref{eq:FB_minus_one_zero}.  Thus $(h,0)$ is not a minimizer, and the positivity of $\psione$ in $SG_h$ is achieved. 

\end{proof}

Now, we show (1) of Theorem \ref{teo:main1};
moreover, we prove  in particular that the (symmetric) 
convex function $h$ cannot touch
and detouch the value $-1$.

\begin{lemma}\label{lemma_reg2_bis}
	Suppose that $(\h,\psione)$ is a 
	minimizer of \eqref{eq:B_intro} such 
	that:
	\begin{itemize}
		\item[(i)]
		$h$ is not identically $-1$;
		\item[(ii)]
		$\psione$ is
		symmetric with respect to $\{\axialcoordofcylinder=\longR\} \cap \doubledrectangle$.
	\end{itemize}
	Then $$h(\axialcoordofcylinder)>-1
	\qquad \forall\axialcoordofcylinder\in [0,2\longR].
	$$
\end{lemma}
\begin{proof}
Since $h \in \Hspace$, it is symmetric with respect to 
	$\{\axialcoordofcylinder=\longR\} \cap \doubledrectangle$; 
hence, by assumption (ii), 
we may restrict our argument to $[0,\longR]$.
	Assume by contradiction that there exists $\overline w_1 
	\in (0,\longR]$ such that $\h( \overline w_1)=-1$. 
	Recall that $h$ is convex,
	nonincreasing in $[0,\longR]$ and continuous at $\longR$.
	Let $$w_1^0:=
	\min\{\axialcoordofcylinder \in (0,\longR]:h(\axialcoordofcylinder)=-1\}.$$ 
By assumption (i)
	we have $w_1^0>0$ and, by convexity,
$h$ is strictly decreasing in $(0,w_1^0)$. 
	We have,
	using \eqref{eq:FA} and \eqref{eq:partial_D_G_h},
\begin{equation}\label{eq:first_rewritten_form}
	\begin{aligned}
		&
		\frac12\FB(\h,\psione)
		= \frac12\left[ \Aone(\psione, \Omegah)
		+
		\int_{\partialbar \Omegah}  \vert \psione  -\dirdatum\vert ~d \Hone 
		+ 
		\int_{\graphh \setminus \{w_2=-1\}}\vert \psione^- \vert ~d \Hone 
		+
		\int_{\Lh} \dirdatum ~d \Hone \right] 
		\\
		= & \frac12 \Aone(\psione, \Omegah)
		+
		\int_{(-1,\h(0))}  \vert \psione(0,\scoord)
		-\dirdatum(0,\scoord)\vert ~d \scoord 
		+
		\int_{(0,w_1^0)}  \vert \psione 
		(\axialcoordofcylinder,-1)  -\dirdatum(\axialcoordofcylinder,-1)\vert 
		~d \axialcoordofcylinder 
		\\
		&\qquad
		\ \ \qquad 
		+\int_{G_{h{\llcorner{(0,w_1^0)}}}}
		\vert \psione^- \vert d \Hone
		+
		\int_{(\h(0),1)} \dirdatum(0,\scoord) d \scoord.
	\end{aligned}
\end{equation}
Recalling Proposition \ref{prop:reg}, we have 
$$
		\frac12\FB(\h,\psione) = 
			\int_{0}^{w_1^0}\int_{-1}^{h(\axialcoordofcylinder)} 
			\sqrt{1+\vert \nabla \psione\vert^2}~
			d\scoord d\axialcoordofcylinder.
$$	
Now, 
	we argue by slicing 
	the rectangle $R_l=(0,\longR)\times (-1,1)$
	with lines $\{\axialcoordofcylinder=\tau\}, \tau \in (0,\longR)$. 
	Recalling the expression of $\Omegah$ (which is non empty 
	by assumption (i)), and neglecting the third addendum in 
\eqref{eq:first_rewritten_form},
	\begin{equation}\label{eq:slicing}
		\begin{aligned}
			\frac12\FB(\h,\psione)
			=&   
			\int_{0}^{w_1^0}\int_{-1}^{h(\axialcoordofcylinder)} 
			\sqrt{1+\vert \nabla \psione\vert^2}~
			d\scoord d\axialcoordofcylinder
			\\
			&+ 
			\int_{(-1,\h(0))}  \vert \psione(0,\scoord)  
			-\dirdatum(0,\scoord)\vert d \scoord  +
			\int_{(0,w_1^0)}  
			\vert \psione (\axialcoordofcylinder,-1)  
			-\dirdatum(\axialcoordofcylinder,-1)\vert ~d \axialcoordofcylinder 
			\\
			&
			+\int_{G_{h{\llcorner{(0,w_1^0)}}}}\vert \psione^- \vert d \Hone
			+
			\int_{(\h(0),1)} \dirdatum(0,\scoord) d \scoord
			\\ 
			\geq 
			&\int_{0}^{w_1^0}\int_{-1}^{h(\axialcoordofcylinder)} \sqrt{1+\vert\nabla \psione\vert^2}
			d\scoord d\axialcoordofcylinder
			+
			\int_{(-1,\h(0))}  \vert \psione(0,\scoord)  -\dirdatum(0,\scoord)\vert d \scoord 
			\\
			&+ \int_{G_{h{\llcorner{(0,w_1^0)}}}}\vert \psione^- \vert d \Hone
			+
			\int_{(\h(0),1)} \dirdatum(0,\scoord) d \scoord
			\\
			> & 
			\int_{0}^{w_1^0}
			\int_{-1}^{h(\axialcoordofcylinder)} \vert \grad_\axialcoordofcylinder 
			\psione (\axialcoordofcylinder,\scoord) \vert 
			d\scoord d\axialcoordofcylinder
			+
			\int_{(-1,\h(0))}  \vert \psione(0,\scoord) 
			-\dirdatum(0,\scoord)\vert d \scoord 
			\\
			&+ \int_{G_{h{\llcorner(0,w_1^0)}}}\vert \psione^- \vert d \Hone
			+
			\int_{(\h(0),1)} \dirdatum(0,\scoord) d \scoord,
		\end{aligned}
	\end{equation}
	where $\grad_{w_1}$ stands for the partial derivative with respect to $w_1$.
	
	Now, let 
	$$h^{-1}:[-1,\h(0)]\rightarrow [0,w_1^0]$$
	be the inverse of $h\llcorner
[0,w_1^0]$.
	Neglecting  $\sqrt{1+(\frac{d}{dw_2} h^{-1})^2}$ in the third addendum on the
	right-hand side of \eqref{eq:slicing}, using also that $\varphi \geq 0$ and
$\psi \geq 0$ (Proposition \ref{prop:reg}), 
	we deduce
	\begin{equation*}
		\begin{aligned}
			\frac12\FB(\h,\psione)
			> 
			&\int_{-1}^{\h(0)}\int_{0}^{\h^{-1}(\scoord)} \vert 
			\grad_\axialcoordofcylinder  \psione (\axialcoordofcylinder,\scoord) 
			\vert ~
			d\axialcoordofcylinder
			d\scoord 
			+ 
			\int_{(-1,\h(0))}  \vert \psione(0,\scoord)  -\dirdatum(0,\scoord)\vert d \scoord 
			\\
			& +\int_{(-1,h(0))}\psione^- (h^{-1}(\scoord),\scoord)d\scoord 
			+\int_{(\h(0),1)} \dirdatum(0,\scoord) d \scoord
			\\
			\geq &
			\int_{-1}^{\h(0)}\Big\vert\int_{0}^{\h^{-1}(\scoord)}  
			\grad_\axialcoordofcylinder  \psione (\axialcoordofcylinder,\scoord) 
			d\axialcoordofcylinder \Big\vert d\scoord
			-
			\int_{(-1,\h(0))} \psione(0,\scoord) d \scoord
			\\
			&  +\int_{(-1,h(0))}\psione^- (h^{-1}(\scoord),\scoord)d\scoord 
			+\int_{(-1,1)} \dirdatum(0,\scoord) d \scoord
			\\
			\geq &
			\int_{(-1,\h(0))}\lvert  \psione (\h^{-1}(\scoord),\scoord) -\psione (0,\scoord)\rvert d\scoord
			-
			\int_{(-1,\h(0))} \psione(0,\scoord) d \scoord 
			\\
			& +\int_{(-1,\h(0))}\psione^- (h^{-1}(\scoord),\scoord)d\scoord 
			+ \int_{(-1,1)} \dirdatum(0,\scoord) d \scoord
			\nonumber\\
			\geq & \int_{(-1,1)}\varphi(0,\scoord)d\scoord =\frac12\FB(-1,0).
			\nonumber 
		\end{aligned}
	\end{equation*}
	Hence the value of $\FB$ on 
	the pair $(\h\equiv-1, \psi \equiv 0)$  
	is smaller than $\FB(\h,\psi)$, thus
	contradicting the minimality of $(h,\psi)$. 
\end{proof}

The next lemma concludes, in particular, the proof of the first statement 
in Theorem \ref{teo:main1}.
\begin{lemma}[\textbf{Symmetry of minimizers}]\label{lem:symmetry_min}
 Every minimizer $(\h,\psione)$ of \eqref{eq:B_intro} is 
such that 
$\psione$ is
 symmetric with respect to $\{\axialcoordofcylinder=\longR\} \cap \doubledrectangle$.
\end{lemma}
\begin{proof}
Let $I \subset (0,2\longR)$ be an open interval; consistently
with \eqref{eq:F_2l}, and since $\psi$ is continuous in $SG_h$, we set
\begin{align*}
\FB(\h, \psione; I):= &	
\Aone(\psione, I\times(-1,1))-
\Htwo\Big(I\times(-1,1) \setminus \Omegah\Big)+
\int_{(\partialbar\doubledrectangle) \cap (\overline I \times [-1,1)) }  \vert \psione  -\dirdatum\vert ~d \Hone \nonumber
\\
&+ \int_{(\partial \doubledrectangle \setminus \partialbar\doubledrectangle) 
\cap (\overline I \times (-1,1])}  
\vert \psione \vert~ d\Hone.
\end{align*}  
Recall that $\h\in \Hspace$, hence  its graph 
is symmetric with respect to $\{\axialcoordofcylinder=\longR\}
\cap \doubledrectangle$. 
Define
$\tilde \psione := \psione $ on $(0,\longR)\times (-1,1)$ 
and  $\tilde \psione (\axialcoordofcylinder,\scoord) := \psione (2\longR -
\axialcoordofcylinder,\scoord)$ 
for  $(\axialcoordofcylinder,
\scoord) \in (\longR, 2\longR)\times (-1,1)$, in particular 
the graph of 
$\tilde\psione$ is symmetric with 
respect to $\{\axialcoordofcylinder=\longR\} \cap 
\doubledrectangle$. 
Since $\FB 
(\h , \psione;(0,\longR))= 
\FB
(\h ,\psione;(\longR,2\longR))$, it follows 
$\FB(\h , \tilde\psione)= \FB(\h, \psione)$ 
for, if $\FB 
(\h , \psione;(0,\longR))<\FB
(\h ,\psione;(\longR,2\longR))$, then 
$\FB(\h , \tilde\psione)< \FB(\h, \psione)$ which contradicts the 
minimality of $(h,\psione)$. 

Now, if $h\equiv-1$ then the minimizer $\psi=0$ is symmetric. On the other hand, if $h$ is not 
identically $-1$, due to Lemma \ref{lemma_reg2_bis}, we have that $h(l)>-1$ so that $SG_h$ is open and connected, and thus the two analytic functions $\psi$ and $\tilde \psi$, coinciding on $SG_h\cap R_l$, 
must coincide. Hence, $\psi=\tilde \psi$ and $\psi$ is symmetric.
\end{proof}

Now, we prove items (2i) and (2iii) of Theorem \ref{teo:main1}:
 the proof will be a consequence of 
the next lemma and Theorem \ref{teo:boundary_regularity}. Recall the definition
of $\partial_D SG_h$ in \eqref{eq:partial_D_G_h}.

\begin{lemma}\label{lem:attaining}
 Let $(h,\psi)$ be a minimizer of \eqref{eq:B_intro} with $h$ not identically $-1$. 
Then $\psi$ attains the boundary condition on 
$\partial_D\Omegah$.
\end{lemma}
\begin{proof}
The result follows from \cite[Theorem 15.9]{Giusti:84},  
since $\partial_D\Omegah$ is union of three segments. 
\end{proof}

\begin{remark}\label{rem:h_id_1}
 In the hypotheses of Lemma 
\ref{lemma_reg2_bis}, if $h\equiv1$ then the graph of $h$ is a 
segment and, as in Lemma \ref{lem:attaining},  $\psi=0$ on $G_h$.
\end{remark}

The conclusion of the proof of Theorem \ref{teo:main1} (2iii)
is given
by the following delicate result.

\begin{theorem}[\textbf{Boundary regularity}]\label{teo:boundary_regularity} Assume there is a minimizer 
$(h,\psi) \in 
X_{2\longR}^{{\rm conv}}$ 
of \eqref{eq:B_intro} with $h$ not identically $-1$. 
Then there exists another minimizer 
$(\widetilde h,\widetilde\psi) \in 
X_{2\longR}^{{\rm conv}}$ 
of 
\eqref{eq:B_intro} having the following properties:
\begin{itemize}
\item[(i)] 
$\widetilde h(0) =1 =
\widetilde h(1)$,
\item[(ii)]
 $\widetilde \psi$ is continuous up
to the boundary of $SG_{\widetilde h}$, 
\item[(iii)]
$\widetilde\psi = 0 \quad{\rm on}~ G_{\widetilde h}$.
\end{itemize}
\end{theorem}

\begin{proof}
By Remark \ref{rem:h_id_1}, we can assume that $h$ is not
 identically $1$ and, by Lemma \ref{lemma_reg2_bis}, also that 
$h(\axialcoordofcylinder) \geq h(l)>-1$ for any $\axialcoordofcylinder\in 
[0,2\longR]$.
 We start to fix a number $\bar s \in (-1,h(l))$ and
to set
$$
K:=(0,2\longR)\times (\bar s, 1) \subset R_{2\longR}.
$$
The usefulness of $\overline s$ stands on the fact that,
by \eqref{eq:psi_is_positive} and Lemma \ref{lem:attaining}, 
the graph of the restriction of $\psi$ over 
$\partial K\setminus \{w_2=1\}$ is {\it strictly positive}
(in particular, excluding the 
two points $(0,1)$ and $(2\longR,1)$). We shall see
at the end of the arguments, that the proof will be 
independent of the choice of $\overline s$.

Let us extend $\psi$ in  $\R^2 \setminus R_{2\longR}$ as follows:
we define $\extpsi : \R^2 \to [0,1]$, $\extpsi := \psi$ in $R_{2\longR}$, and 
\begin{equation}\label{eq:widehat_psi}
 \extpsi(
\axialcoordofcylinder,\scoord):=\begin{cases}
                \varphi(w_1,\scoord) &\text{if } \axialcoordofcylinder
<0 \text{ or } \axialcoordofcylinder>2\longR, \text{ and }|\scoord|\leq 1,\\
                0&\text{if }|\scoord|>1,
               \end{cases}
\end{equation}
 (see \eqref{eq:varphi}).
In this way  $\extpsi$ is continuous in 
$\R^2\setminus \overline R_{2\longR}$.

Now, we divide the proof
into eight steps.
In step 1 we start by regularizing
$\extpsi$ 
in order that the regularized functions have smooth graphs over the sets $K_n$ 
defined in \eqref{eq:enlarged_rectangle}, and so
these graphs are of 
disc-type. We expect the graph 
of $\extpsi$ over the sets $K_n$, considering
also a possible vertical part over the graph of $h$, to be
a surface of disc-type; however, we miss the proof
of this fact, mainly due to possible irregularity of the trace of $\extpsi$
over $G_h$. The information on the topological type of 
these graphs will be crucial in our proof. 
 
Also, an appropriate approximation of $h$ will be
needed; this latter approximation depends on the
approximation of $\widehat \psi$.
Next (step 2), 
we will compare these graphs with the solution of 
a suitable disc-type Plateau problem.

\medskip

\textit{Step 1, part 1: Approximation of $\extpsi$.} 

Let $n>0$ be a natural number
  (that will be sent to $+\infty$ later) 
such that $\bar s+\frac1n<h(l)$,
and consider the enlarged rectangle
\begin{equation}\label{eq:enlarged_rectangle}
K_n:=\left(-\frac1n,2l+\frac1n\right)\times \left(\bar s, 1+\frac1n\right),
\end{equation}
see Fig. \ref{fig:curves_Frechet}.
Note that 
\begin{equation}\label{eq:widehat_psi_is_continuous_on_partial_K_n}
\extpsi {\rm ~ is~ continuous~
on~ } \partial K_n.
\end{equation}
Given $n \in \mathbb N$, we claim that 
we can build a sequence 
$(\psi_k^n)_{k\in \mathbb N}$ depending on $n$,
which satisfies the following properties:
\begin{equation}\label{eq:properties_of_psi_n}
\begin{aligned}
&\psi_k^n \in C^\infty(K_n,[0,1]) \cap 
C(\overline K_n,[0,1]) \qquad \forall k \in \NN,~ k>0,
\\
 &\psi_k^n=\extpsi \text{ on }\partial K_n\qquad\forall k\in \NN, ~ k>0,
\\
 &\psi^n_k\rightharpoonup \extpsi \text{ weakly}^\star \text{ in }BV(K_n)
\text{ as } k\rightarrow +\infty,
\\
& \int_{K_n}|\grad \psi^n_k|~dw \rightarrow |D\extpsi|(K_n)
\text{ as } k\rightarrow +\infty.
\end{aligned}
\end{equation}
In order to obtain
\eqref{eq:properties_of_psi_n}
we use standard arguments 
(details can be found in \cite[Thm. 3.9]{AmFuPa:00} or \cite[Thm. 1, Section 4.1.1]{GiMoSu:98}). 
To the aim of our discussion, we just recall that we proceed by constructing an 
increasing sequence 
$(U_\ivirgolan
)_{i\geq 1}$ of open
subsets of $K_n$,
 $U_\ivirgolan \subset \subset U_\ipiuunovirgolan
\subset \subset K_n$, $\cup_i U_\ivirgolan
= K_n$ (for 
$i \geq 1$ we take
$U_\ivirgolan:=\{x\in\R^2:\textrm{dist}(x,\R^2\setminus K_n)>
\frac{1}{i+n}\}$
for definitiveness) and with the aid of a
partition  
of unity
$(\eta_\ivirgolan)$ 
associated to $V_\unovirgolan:=U_\duevirgolan$, 
$V_\ivirgolan
:=U_\ipiuunovirgolan
\setminus \overline U_\imenounovirgolan$ for $i\geq2$, 
we 
mollify $\widehat \psi$ accordingly in $V_\ivirgolan$.
For our purpose we 
choose\footnote{
We need the full set $\overline V_\ivirgolan$ as support 
in order that the argument to detect the behaviour of $h_n$ 
(defined in \eqref{eq:def_h_n}) in 
$[-\frac1n,0]$ applies.} 
$\eta_\ivirgolan$ in such a way that 
\begin{align}\label{support_eta_i}
\supp(\eta_\ivirgolan)=\overline V_\ivirgolan.
\end{align}
Since $\psi_k^n$ is obtained by mollification
 we have $\psi_k^n\in 
C^\infty(K_n)$ and moreover 
$\psi^n_k\in  C(\overline K_n)$ 
because it attains the continuous boundary datum $\extpsi$
on $\partial K_n$.
Here we use the same standard 
mollifier $\rho\in C_c^\infty (\unitdisc)$ in each $V_\ivirgolan$, choosing,
for $w = (w_1,w_2)$, $\rho_{\ivirgolan,k}(w):=\rho(w/r_{\ivirgolan,k})$ with  
$r_{\ivirgolan,k}:=r_\ivirgolan/k>0$,  $r_\ivirgolan$ decreasing with 
respect to $i\geq1$, with
$r_\ivirgolan\rightarrow 0^+$ as $i\rightarrow+\infty$;
we take 
\begin{equation}\label{eq:choice_of_radius}
r_\ivirgolan=\frac{1}{i+2+n}
\end{equation}
for definiteness.    
Finally, $[0,2l]\times[\bar s+\frac1n,1]\subset U_\unovirgolan\subset V_\unovirgolan$, 
and $V_\ivirgolan\cap \left(
[0,2l]\times[\bar s+\frac1n,1]\right)=\emptyset$ for $i \geq 2$.
It follows 
\begin{equation}\label{mollification_in_K}
\psi_k^n=\widehat\psi\star \rho_{\unovirgolan,k}\qquad \text{ in }[0,2l]\times\Big[
\bar s+\frac1n,1\Big] \qquad \forall
n \in \mathbb N. 
\end{equation}
Using \cite[Prop. 3 Sec. 4.2.4 pag. 408, and Th. 1 Sec. 4.1.5 pag. 331]{GiMoSu:98} 
we infer
\begin{align}\label{14.34}
 \mathcal A(\psi_k^n,K_n)\rightarrow \areaonecod(\extpsi, K_n)
\qquad \text{ as } k\rightarrow +\infty. 
\end{align}
Now that properties \eqref{eq:properties_of_psi_n} are achieved, 
by a diagonal argument we select 
functions 
\begin{equation}\label{eq:psi_n}
\psi_{n}:=\psi_{k_n}^n\in (\psi^n_k) \qquad \forall n \in \NN,
\end{equation}
such that
\begin{equation}\label{eq:diag}
\begin{aligned}
 &\psi_n=\extpsi \text{ on }\partial K_n \qquad \forall n\in\mathbb N,
\\
 &\psi_n\rightharpoonup \extpsi \text{ weakly}^* \text{ in }BV(K)
\text{ as } n\rightarrow +\infty,\\
& \int_{K_n}|\grad \psi_n|~dw\rightarrow |D\extpsi|(\overline{K})
\text{ as } n\rightarrow +\infty,
\end{aligned}
\end{equation}
where $\overline K$ is the closed rectangle $\overline K := 
\cap_n K_n$.
On the basis of \eqref{14.34} and  \eqref{eq:diag}, 
we can also ensure\footnote{To prove claim \eqref{conv_areas_Kn}, 
fix $m\in \mathbb N$, and set 
$\tilde \psi_n:=\extpsi$ outside $K_n$ and $\tilde\psi_n = \psi_n$
in $K_n$, so 
that 
\begin{align*}
&\tilde \psi_n\rightharpoonup \extpsi \text{ weakly}^* \text{ in }BV(K_m)
\text{ as } n\rightarrow + \infty,
\\
& |\grad \tilde \psi_n|(K_m)\rightarrow|D\extpsi|(K_m)= |D\extpsi|(\overline{K})+|D\extpsi|(K_m\setminus \overline{K})
\text{ as } n\rightarrow + \infty.
\end{align*}
Then   
$\limsup_{n\rightarrow+\infty} \areaonecod(\psi_n, K_n)\leq 
\limsup_{n\rightarrow+\infty} \areaonecod(\tilde \psi_n, K_m)= \areaonecod(\extpsi, K_m)
= \areaonecod(\extpsi, \overline{K})+ \areaonecod(\extpsi, K_m\setminus \overline{K})$,
the first equality following from the strict convergence 
of $\tilde \psi_n$ to $\widehat \psi$
\cite[Prop. 3 Sec. 4.2.4 pag. 408 and 
Thm. 1 Sec. 4.1.5 pag. 371]{GiMoSu:98}. 
Taking the limit as $m\rightarrow +\infty$, 
since 
$\widehat \psi \in W^{1,1}(K_m\setminus \overline{K})$
we conclude
$\limsup_{n\rightarrow+\infty} \areaonecod
(\psi_n, K_n)\leq \areaonecod(\psi, \overline{K})$.
Then 
\eqref{conv_areas_Kn} follows by 
lower semicontinuity.} 
 that
\begin{align}\label{conv_areas_Kn}
 \mathcal A
(\psi_n, K_n)\rightarrow \areaonecod(\extpsi, \overline K) \qquad{\rm 
as}~ n \to +\infty.
\end{align}
Here, by $ \areaonecod(\extpsi, \overline K)$ we mean the 
area of the graph of $\extpsi$ relative to $\overline K$
which, recalling also 
Proposition \ref{prop:reg},
 reads as
\begin{equation}\label{eq:area_closure}
\areaonecod(\extpsi, \overline K)=\areaonecod(\extpsi, K)+
\int_{\{0\}\times(\bar s,1)}|\extpsi^--\varphi|~d\mathcal H^1+
\int_{\{2l\}\times(\bar s,1)}|\extpsi^--\varphi|~d\mathcal H^1, 
\end{equation}
where $\extpsi^-$ denotes 
the trace of $\extpsi$ on $\partial K$. This concludes the proof 
of the first part of step 1.

Before passing to the second part, for any $n \in \NN$
we define
$$
\widehat h(\axialcoordofcylinder):=
\sup\left\{\scoord\in \Big(\bar s,1+\frac1n\Big):
\widehat \psi(\axialcoordofcylinder,\scoord)>0\right\}
\qquad \forall \axialcoordofcylinder
\in \Big(-\frac1n,2l+\frac1n\Big).
$$
Notice that 
\begin{equation*}
\begin{aligned}
& \widehat h = h \quad{\rm in}~ [0,2l],
\\
& \widehat h = 1 \quad{\rm in}~
(-1/n,0) \cup (2l, 2l + 1/n).
\end{aligned}
\end{equation*}

\medskip

\textit{Step 1, part 2: Approximation of $h$.} 
We construct functions
$h_n: (-\frac1n,2l+\frac1n) \to (\bar s, 1+\frac1n)$ 
such that
\begin{equation}\label{eq:constr_h_n}
\begin{aligned}
& h_n(\cdot) = h_n(2\longR - \cdot),
\\
&  \psi_n=0 ~{\rm in}~ K_n \setminus SG_{h_n},
\\
& h_n \in BV\big(\big(-\frac{1}{n}, \longR\big) \big),
\\
\end{aligned}
\end{equation}
and, setting
\begin{align}\label{topology_Kn+}
SG_{h_n,\overline s}:=\left\{(\axialcoordofcylinder,\scoord):
\axialcoordofcylinder\in \left(-\frac1n,2l+\frac1n\right),\; \scoord
\in (\bar s,h_n(\axialcoordofcylinder))\right\},
\end{align}
also such that
\begin{equation*}
\begin{aligned}
& \lim_{n \to +\infty} \mathcal H^2(SG_{h_n, \overline s})
=\mathcal H^2(K \cap \Omegah),
\\
& \lim_{n \to +\infty}\areaonecod(\psi_n,  SG_{h_n, \overline s}) = 
\mathcal F_{2l}(h,\psi)-\areaonecod(\psi, \doubledrectangle\setminus K).
\end{aligned}
\end{equation*}

To this aim, for any $n\in \mathbb N$ 
we define
\begin{equation}\label{eq:def_h_n}
 h_n(\axialcoordofcylinder):=
\sup\left\{\scoord\in \Big(\bar s,1+\frac1n\Big):
\psi_n(\axialcoordofcylinder,\scoord)>0\right\}
\qquad \forall \axialcoordofcylinder
\in \Big(-\frac1n,2l+\frac1n\Big),
\end{equation}
Since (see \eqref{eq:psi_is_positive}
of Proposition \ref{prop:reg} and \eqref{eq:widehat_psi}) 
$\extpsi$ is positive in $\Omegah\cup((-\frac1n,0)\times (\bar s,1))\cup((2l,2l+\frac1n)\times (\bar s,1))$ 
it turns out, recalling also that the function 
$\psi_n$ in \eqref{eq:psi_n} is obtained by mollification, that

\begin{equation}\label{eq:lines}
\begin{aligned}
 & -1 < h(\axialcoordofcylinder)<h_n(\axialcoordofcylinder)<1+\frac1n \qquad 
\forall \axialcoordofcylinder\in (0,2l),
\\
 &1<h_n(\axialcoordofcylinder)<1+\frac1n \qquad \forall
\axialcoordofcylinder\in \Big(-\frac1n,0\Big]\cup\Big[2l,2l+\frac1n\Big).
\end{aligned}
\end{equation}
The validity of \eqref{eq:lines} is due to the fact that 
$\psi$ is positive in the subgraph of 
$\widehat h$ and vanishes on the epigraph of $\widehat h$. Therefore,
when mollifying $\psi$, the positivity set
must increase (and the mollified function must vanish 
at points at distance from the subgraph of $\widehat h$ of the order
of the mollification radius. Thus $h_n > h$;
also 
$h_n < h+\frac1n$ due to our choice of $r_{i,n}$ in \eqref{eq:choice_of_radius},
since the mollification radius is smaller than
$\frac1n$. 

Moreover, again the positivity of $\extpsi$ implies
that 
\begin{align}\label{topology_Kn+_bis}
\psi_n>0 \ \  {\rm ~in}~
 \ \ SG_{h_n,\overline s} \subset K_n,
\end{align}
whereas
\begin{align}\label{eq:psi_n_zero}
 \psi_n(\axialcoordofcylinder, \scoord)=0 \qquad \text{if }
\tcoord \in \Big(-\frac{1}{n}, 2 \longR + \frac{1}{n}\Big), 
~\scoord\in \Big[h_n(\axialcoordofcylinder),1+\frac1n\Big),
\end{align}
because $\extpsi(\axialcoordofcylinder,
\scoord)=0$ 
if $\axialcoordofcylinder\in[0,2l]$, $\scoord>h(\axialcoordofcylinder)$ and if $\scoord>1$.
Exploiting \eqref{mollification_in_K}, and the fact that $h$ is 
nonincreasing 
(resp. nondecreasing)  in $[0,\longR]$ (resp. in $[l,2l]$), one checks\footnote{
Let us show for instance that $h_n$ is decreasing in 
$[0,\longR]$.
Recall that the function $\widehat \psi$ vanishes above the graph of $h$, 
which is decreasing in $[0,\longR]$. Now, take a point $(w_1,w_2)\in K_n$, $w_1 
\in [0,\longR)$,  $w_2 > h(w_1)$; suppose first that $w_1 \geq r_\unovirgolan$.
If ${\rm dist}((w_1,w_2), {\rm graph}(h))
> r_\unovirgolan$, then $\psi_n(w_1,w_2)
=\widehat \psi\star \rho_\unovirgolan(w_1,w_2)
=0$, and  
if ${\rm dist}((w_1,w_2), {\rm graph}(h))
< r_\unovirgolan$, then $\psi_n(w_1,w_2)=\widehat \psi\star \rho_\unovirgolan(w_1,w_2)
>0$.
Hence,  if $\widehat \psi\star \rho_\unovirgolan(w_1,w_2)=0$
then also $\widehat \psi\star \rho_\unovirgolan(w_1+\eps,w_2)=0$ 
for $\eps>0$ small enough, because 
${\rm dist}((w_1+\eps,w_2), {\rm graph}(h)) >
{\rm dist}((w_1,w_2), {\rm graph}(h))$, 
being $h$ decreasing in $[0,\longR]$.
This argument applies also when $w_1 \in [0,r_\unovirgolan)$
by \eqref{mollification_in_K}, since 
$\overline h$ is nonincreasing also in $(-1/n, l)$.
}
that also $h_n$ 
is nonincreasing in  $[0,\longR]$
(resp. nondecreasing in $[l,2l]$). 
Concerning the behaviour of $h_n$ in $(-\frac1n,0]$ (and similarly in $[2l,2l+\frac1n)$), we see that in $V_\ivirgolan$ ($i>1$), we are mollifying 
with $\rho_{\ivirgolan,k_n}$ 
whose radius of mollification is $r_\ivirgolan/k_n$, so that 
$\widehat \psi\star \rho_{\ivirgolan,k_n}$  equals $0$ 
on the line $\{w_2=1+\frac{r_\ivirgolan}{k_n}\}$, and nonzero below inside $K_n$:
this follows from the fact that $\widehat \psi$ is  $0$
 on the line $\{w_2=1\} \cap K$ and nonzero below.
We have defined the radii $r_\ivirgolan$ 
to be decreasing 
with respect to $i$, so that, $\psi_n$ being the 
sum
of $\widehat \psi\star \rho_{\ivirgolan,k_n}$  (whose support is 
${\overline V}_\ivirgolan$ by \eqref{support_eta_i}), 
it turns out that $\psi_n$ is $0$ on $\{w_2=1+\frac{r_\ivirgolan}{k_n}\}$ 
and nonzero below\footnote{Notice that in $V_\ivirgolan
\setminus V_\imenounovirgolan$ only $\widehat \psi\star \rho_{\ivirgolan,k_n}$ and 
$\widehat \psi\star \rho_{\ipiuunovirgolan,k_n}$, are nonzero (from this it follows that $h_n=1+\frac{r_\ivirgolan}{k_n}$ 
in $(-\frac1n+\frac{1}{i+n+1},-\frac1n+\frac{1}{i+n}]$). }
in $V_\ivirgolan\setminus V_\imenounovirgolan$. 
As a consequence, $h_n$ is nondecreasing\footnote{ 
		Precisely,  $h_n$ is piecewise constant and nondecreasing 
	in $(-\frac1n,0]$, but these two properties are not 
needed in the proof.
} 
	 in $(-\frac1n,0)$. 
In particular
\begin{equation}\label{h_n_BV}
h_n\in BV\Big((-\frac1n,2l+\frac1n)\Big). 
\end{equation}
Finally, it is not difficult to see that the functions $h_n$ 
converge 
to $h$ in $L^1((0,2l))$ 
as $n\rightarrow \infty$, 
and
\begin{align}
\lim_{n \to +\infty}
 \mathcal H^2(SG_{h_n,\overline s}) = H^2(K \cap 
\Omegah).
\label{eq:25}
\end{align}
From this, \eqref{conv_areas_Kn}, Lemma \ref{lem:attaining}, \eqref{eq:area_closure} and
\eqref{eq:F_2l}
we deduce
\begin{align}\label{convergenza_plateau}
\areaonecod(\psi_n,  SG_{h_n,\overline s})=
 \areaonecod(\psi_n, K_n)-\mathcal H^2(K_n\setminus SG_{h_n, \overline s})
\rightarrow\mathcal F_{2l}(h,\psi)-\areaonecod(\psi, \doubledrectangle\setminus K).  
\end{align}

\textit{Step 2: The curves $\Gamma_n$, and 
the surfaces $\Sigma_n$ and $\mathcal G_{\psi_n}$. Comparison with a Plateau problem.} 

In this step we 
compare the graph of $\psi_n$ over $K_n$
with the solution of a
 disc-type Plateau problem. In particular we will 
obtain a disc-type surface $\dtsnp$ 
whose area is smaller than 
or equal to the area of the graph of $\psi_n$,
see
\eqref{key_ineq_D+}.
 In step 3 (see \eqref{ineq_withF}) we will compare this surface with the graph of $\psi$ on $K$. 

We recall that $\psi_n$ is continuous in 
$\overline K_n$, it is positive on the bottom edge 
$[-\frac1n,2l+\frac1n]\times \{\bar s\}$
of $K_n$ (see \eqref{eq:diag}), 
it is zero on the top edge $[-\frac1n,2l+\frac1n]\times \{1+\frac1n\}$
by \eqref{eq:lines}, and on the lateral edges of $K_n$ it coincides with $\widehat \psi$; more specifically
\begin{equation*}
\begin{aligned}
 &\psi_n\Big(-\frac1n,\scoord\Big)=\psi_n\Big(2l+\frac1n,\scoord\Big)
=\varphi(0,\scoord)>0\qquad \text{for } \scoord\in [\bar s,1),\nonumber\\
 &\psi_n\Big(-\frac1n,\scoord\Big)
=\psi_n\Big(2l+\frac1n,\scoord\Big)=0\qquad \text{for }\scoord\in \Big[1,1+\frac1n\Big).
\end{aligned}
\end{equation*}
Define
$$\partial_DK_n:=\Big(\Big[-\frac1n,2l+\frac1n\Big]
\times \{\bar s\}\Big)\cup\Big(\Big\{-\frac1n,2l+\frac1n\Big\}\times [\bar s,1]\Big).
$$
{}From \eqref{eq:diag}, we see that $\psi_n$ coincides with 
$\extpsi$ over  $\partial_DK_n$, and its graph over this set is a 
curve,
that we 
denote by $\Gamma_n^+$. This curve, excluding its endpoints 
$P_n=(-\frac1n,1,0)$ and $Q_n=(2l+\frac1n,1,0)$, 
is contained in the half-space 
$\{w_3>0\}$, while 
$P_n, Q_n \in \{w_3=0\}$.
We further denote by $\Gamma_n^-$ the
 symmetric of $\Gamma_n^+$ with respect to the plane $\{w_3=0\}$, 
so that 
$$
\Gamma_n:=\Gamma_n^+\cup\Gamma_n^-
$$
is a Jordan curve in $\R^3$, see
Fig. \ref{fig:curves_Frechet}.
Thus we can solve the disc-type
Plateau problem with boundary $\Gamma_n$ \cite{Hil1} and 
call $\dtsn \subset \R^3$ one of its  
solutions\footnote{$\dtsn$ is the image of an area-minimizing
map from the unit disc into $\R^3$.}. 
In addition, we may assume that $\dtsn$ is symmetric with respect to 
the plane $\{w_3=0\}$ and that 
\begin{align*}
 \mathcal H^2(\dtsnp)=\mathcal H^2(\dtsnm),
\end{align*}
with $\dtsn^\pm:=\dtsn\cap \{w_3\gtrless0\}$, respectively (see Fig. \ref{fig:curves_Frechet}).

\begin{figure}
	\begin{center}
		\includegraphics[width=0.8\textwidth]{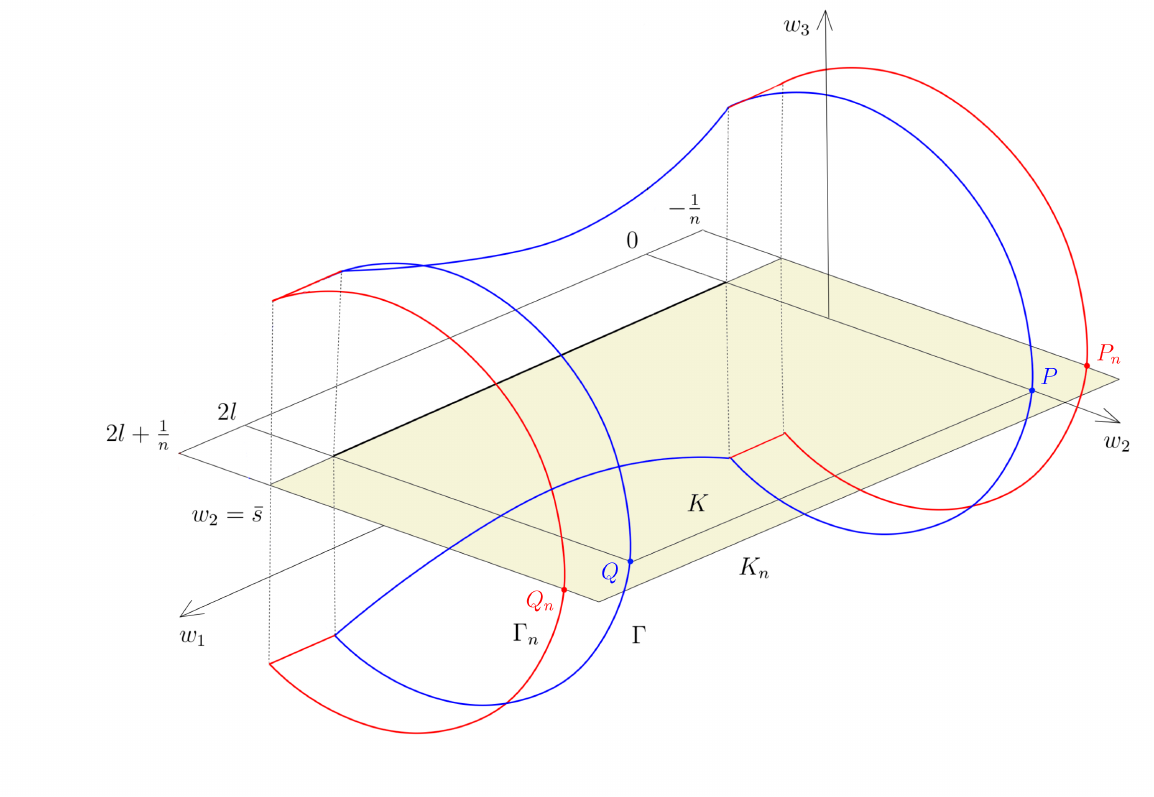}
		\caption{The rectangle $K_n$ 
in \eqref{eq:enlarged_rectangle} in dark, 
and the rectangle $K$ inside. $\Gamma$ is the
curve passing through $Q$ and $P$, the curves $\Gamma_n$
(which pass through $Q_n$ and $P_n$) 
approach $\Gamma$ ($\Gamma$ and $\Gamma_n$ coincide and overlap on the graph of $\psi$ over the bold segment $\{w_2=\bar s\}\cap K$). 
%this part of $\Gamma_n\cap \Gamma$ is in blue
		}
		\label{fig:curves_Frechet}
	\end{center}
\end{figure}

Now, we want to compare the area of the graph
of $\psi_n$ in 
$SG_{h_n, \overline s}$ with 
$\mathcal H^2(\dtsnp)$. 
To this aim we start by observing that $\psi_n$, 
being smooth in $K_n$ and continuous in $\overline K_n$ (see \eqref{eq:properties_of_psi_n}), is such that 
its graph over $SG_{h_n, \overline s}$ has the topology of 
$SG_{h_n, \overline s}$, 
that is the topology of the 
disc.
Indeed, 
$SG_{h_n, \overline s}$ is
bounded by construction, and it is open
from 
\eqref{topology_Kn+_bis}, \eqref{eq:psi_n_zero}. 
In addition, it is connected and simply connected. Indeed,
take any continuous curve $\gamma:S^1\rightarrow SG_{h_n, \overline s}$.
Using \eqref{eq:lines}, 
let $\widehat s\in (\bar s, 1)$ 
be such that $\{w_2=\widehat s\}\cap K_n\subset SG_{h_n, \overline s}$; hence we can 
(vertically) contract $\gamma$ 
continuously to its projection on the line $\{w_2=\widehat s\}$, 
and then contract it continuously to the middle point of $\{w_2=\widehat s\}\cap K_n$, showing that $\gamma$ is homotopic to the constant curve.
 Hence,
by the Riemann mapping theorem, $SG_{h_n, \overline s}$ is biholomorphic to the 
open unit disc, and 
$\overline SG_{h_n, \overline s}$
is homeomorphic to the closure of the disc, thanks to the fact that $\partial SG_{h_n, \overline s}$ is a Jordan curve, due to the 
BV-regularity of $h_n$ (see \eqref{h_n_BV}). 

Denoting by $\mathcal G_{\psi_n}^+$ the graph of $\psi_n$ over $SG_{h_n, \overline s}$, 
we consider the graph $\mathcal G_{\psi_n}^-$  of $-\psi_n$ over $SG_{h_n, \overline s}$, 
and observe that the closure of 
$\mathcal G_{\psi_n}^+\cup \mathcal G_{\psi_n}^-$ is a disc-type surface with boundary $\Gamma_n$. Therefore, 
by minimality,
\begin{align}
\label{key_ineq_D+}
 \areaonecod(\psi_n, SG_{h_n, \overline s})=\mathcal H^2(\mathcal G_{\psi_n}^+)\geq \mathcal H^2(\dtsnp).
\end{align}

\textit{Step 3: Passing to the limit 
as $n\rightarrow+\infty$: the curve $\Gamma$ and the surface $\Sigma$.}

The graph of $\psi$ over the segment $[0,2l]\times \{\bar s\}$ 
and the graph of $\varphi$ over the two segments 
$\{0,2l\}\times [\bar s,1]$ form a simple continuous curve $\Gamma^+$ which,
excluding the two endpoints 
\begin{equation}\label{eq:P_Q}
P=(0,1,0), \qquad Q=(2l,1,0),
\end{equation}
is contained in the half-space $\{w_3>0\}$, while
$P,Q \in \{w_3=0\}$ (see
Fig. \ref{fig:curves_Frechet}).
If we consider
$$
\Gamma:=\Gamma^+\cup\Gamma^-,
$$
with $\Gamma^-$ the symmetric of $\Gamma^+$ with respect to 
the plane $\{w_3=0\}$, a direct check shows that the curves 
$\Gamma_n$ converge to the curve $\Gamma$
in the sense of 
Frechet \cite{Nitsche:89}, as $n\rightarrow +\infty$. 
As a consequence, the area-minimizing 
disc-type surfaces $\dtsn$ defined in step 2 satisfy $\mathcal H^2 (\dtsn)\rightarrow\mathcal H^2(\dts)$ (see \cite[Paragraphs 301, 305]{Nitsche:89}), with 
$\Sigma$ a disc-type area-minimizing surface 
spanned by $\Gamma$. It follows
\begin{equation}\label{eq:convergence_of_areas_Sigma_n}
 \mathcal H^2(\dtsnp)
\rightarrow \mathcal H^2(\dtsp) \qquad {\rm as}~ n \to +\infty,
\end{equation}
where $\dtsp:=\dts\cap \{w_3>0\}$.  
{}From \eqref{eq:convergence_of_areas_Sigma_n}, 
\eqref{key_ineq_D+}, \eqref{convergenza_plateau}
we deduce
$$
\begin{aligned}
\mathcal H^2(\dtsp)  = &
\lim_{n \to +\infty} 
\mathcal H^2(\dtsnp)
\leq 
\lim_{n \to +\infty} 
 \areaonecod(\psi_n, SG_{h_n, \overline s})
\\
=&
\lim_{n \to +\infty}\left(
 \areaonecod(\psi_n, K_n)-\mathcal H^2(K_n\setminus SG_{h_n, \overline s})\right)
=
\mathcal F_{2l}(h,\psi)-\areaonecod(\psi, \doubledrectangle\setminus K).  
\end{aligned}
$$ 
Since $\psi_n = \psi$ on $R_{2\longR} \setminus K$, we 
get 
\begin{align}
\label{ineq_withF}
\lim_{n \to +\infty} \left(
\areaonecod(\psi_n, SG_{h_n, \overline s})+\areaonecod(\psi, \doubledrectangle\setminus K)
\right)=  \FB(h,\psi)\geq \mathcal H^2(\dtsp)+ \areaonecod(\psi, \doubledrectangle\setminus K).  
\end{align}

Let $\Phi=(\Phi_1,\Phi_2,\Phi_3):\overline \unitdisc
\subset \R^2 
\rightarrow \dts \subset \R^3$ be an
analytic and conformal 
 parametrization of $\Sigma$
in the open unit disc
$\unitdisc$, continuous up to 
$\partial \unitdisc$,  with $\Phi(\partial \unitdisc)=\Gamma$. 
Exploiting the results in \cite{Meeks_Yau:82} (see also \cite[pag. 343]{Hil1}) 
we know that %
\begin{equation}
\label{eq:Phi_is_an_embedding}
\Phi {\rm~ is~ an~ embedding}, 
\end{equation}
since $\Gamma$ is a simple curve on
 the boundary of the convex set $K\times \R$.

\medskip
Now, we need to prove several qualitative properties of $\Sigma$:
this will be achieved in steps 4,5 and 6.
\smallskip

\textit{Step 4: 
$\dts \cap \{w_3=0\}$ is a simple curve $\Gamma_0$  
connecting the two points $P$ and $Q$ in \eqref{eq:P_Q}.}

This can be seen as follows:
Assume $\Phi(p_0)=P$ and $\Phi(q_0)=Q$ for two distinct points $p_0,q_0\in \partial \unitdisc$. 
By standard arguments\footnote{See also step 5 where a similar statement is proved.}, 
the open unit disc $\unitdisc$ is splitted into two 
connected components $\{x\in \unitdisc:\Phi_3(x)\geq 0\}$ and $\{x\in \unitdisc:\Phi_3(x)<0\}$ and the set 
$\{\Phi_3=0\}$ must be a simple curve in $\unitdisc$ connecting $p_0$ and $q_0$  (here we use that the points $p_0$ and $q_0$ 
are, by the definition of $\Gamma$ and the properties of $\Phi$, the 
unique points on $\partial \unitdisc$ where $\Phi_3=0$ and that the 
two relatively open arcs on $\partial \unitdisc$ with extreme points $p_0$ and $q_0$ are mapped in $\{w_3>0\}$ and $\{w_3<0\}$ respectively).
By the injectivity of $\Phi$ (property
\eqref{eq:Phi_is_an_embedding}) we conclude that 
\begin{equation}\label{eq:Gamma_0}
\Gamma_0:=\Phi(\{\Phi_3=0\})
\end{equation}
is a simple curve connecting $P$ and $Q$ on the plane $\{w_3=0\}$, 
and more specifically $\Gamma_0 \subset K$. 

\medskip

\medskip

In the next two steps $5$ and $6$ we define
the functions $\widetilde h$ and $\widetilde \psi$ which appear 
in the statement of the theorem.
In step $5$  we show that, due to the particular
shape of $\Gamma$, the surface $\Sigma$ admits a semicartesian
parametrization \cite{BePaTe:16}, namely that if we slice
$\Sigma$ with a plane orthogonal to the first coordinate $w_1
\in (0,2l)$ then the intersection is a curve connecting
the two corresponding points on $\Gamma$; in addition, in this
present case, 
this curve turns out to be simple. We will also show 
that the free part $\Gamma_0$
of $\Sigma$
  leaves a trace on $\doubledrectangle$
which is the graph of a convex function $\widetilde h$ (of one variable). 

\medskip

\textit{Step 5:
The projection $\projthree(\Sigma)$ of $\Sigma$ on 
the plane $\{w_3=0\}$ is the subgraph of a  
function $\widetilde h\in 
\mathcal H_{2l}$,
where we recall that 
$\mathcal H_{2l}$ is defined in \eqref{eq:space_of_h_in_the_doubled_interval}. In particular, $h$ 
is convex.} 

We first show that $\projthree(\Sigma)$ is the subgraph of a 
function $\widetilde h$, and then we prove that 
$\widetilde h\in \mathcal H_{2l}$. 
 Take a point $W = (W_1,W_2,W_3)\in \Sigma\setminus \Gamma$,
$W_1 \in (0, 2l)$; 
by the strong maximum principle, $\projthree(W) \notin \partial K$:
this follows since points in $\Sigma\setminus \Gamma$ are in the interior of the convex envelope of $\Gamma$, see \cite{Hil1}. 
Consider the unique point $x\in \unitdisc$ 
such that $\Phi(x)=W$. 
Due to the particular structure of $\Gamma$,
one checks that 
 $\partial \unitdisc=\Phi^{-1}(\Gamma)$ splits into two 
connected components, $\Phi_1^{-1}((W_1,2l])\cap \partial \unitdisc$ 
and $\Phi_1^{-1}([0,W_1])\cap \partial \unitdisc$,
 since $\Phi_1^{-1}(\{W_1\})\cap \partial \unitdisc$ consists
of two distinct points $q_1$, $q_2$ in $\partial \unitdisc$. 
In particular, the continuous function $\Phi_1(\cdot)-W_1$ changes 
sign only twice on $\partial \unitdisc$, 
namely in correspondence of $q_1$ and $q_2$. 
From Rado's lemma \cite[Lemma 2, pag. 295]{Hil1} 
it follows that there are no points on $\Sigma \cap\{w_1=W_1\}$ 
where the two area-minimizing surfaces $\Sigma$ and the plane $\{w_1=W_1\}$ 
are tangent to each other\footnote{If $\pointP$ is a tangency point, 
then the differential of $\Phi_1$ must vanish at $\Phi^{-1}(\pointP)\in \unitdisc$.}.  
It follows that, if $\pointP\in (\Sigma\setminus \Gamma)\cap \{w_1=W_1\}$, 
then the set $(\Sigma\setminus \Gamma)\cap \{w_1=W_1\}$ is, 
in a neighbourhood of $\pointP$, an analytic curve, 
see again 
\cite[Lemma 2, pag. 295]{Hil1}. 
Hence, $\{\Phi_1=W_1\}\cap \unitdisc$ is, 
in a neighbourhood of $\Phi^{-1}(\pointP)$, an analytic curve.
If $\gamma_I:I\rightarrow \unitdisc$ 
is a parametrization of this curve,
$I=(a,b)$ a bounded open interval, 
we see that the limits as 
$t\rightarrow a^+$ and $t \to b^-$
of 
$\gamma_I(t)$ 
exist\footnote{$\overline \unitdisc$ is compact, 
hence $\gamma_I(t)$ has some accumulation point
as $t\rightarrow a^+$. Notice that $I$ and $\gamma_I(I)$ 
are homeomorphic by contruction; in turn $\gamma_I(I)$ is 
homeomorphic to the analytic curve 
$\Phi\circ\gamma_I(I)$.
Assume $x$ is an accumulation point for $\gamma_I(t)$ as $t\rightarrow a^+$. If $x\in \unitdisc$, there is a neighborhood $U$ of $x$ such that $\sigma:=\Phi(U)\cap \{w_1=W_1\}$ is an analytic curve. Then $\gamma_I$, in a right neighbourhood $J$ of $a$, is homeomorphic to the analytic curve $\Phi\circ\gamma_{I}(J)\in \R^3$ emanating from $\Phi(x)$, which in turn is the restriction of $\sigma$. In particular  $\gamma_I(I)$ is a curve emanating from $x$ and the limit as $t\rightarrow a^+$ of $\gamma_I(t)$ is $x$.
If instead $x\in \partial \unitdisc$ then $x$ must be the unique accumulation point. 
Indeed, 
$\lim_{t\rightarrow a^+}\Phi_1\circ\gamma_I(t)=W_1$, 
and then $x=q_1$ or $x=q_2$, say $x=q_1$.
Assume there is 
another accumulation point $y$ as $t\rightarrow a^+$; then $y\notin \unitdisc$, otherwise we fall in the previous case, 
and therefore necessarily $y=q_2$. But in this case, we 
see that there must be another accumulation point $z\in \unitdisc$ 
(as $t\rightarrow a^+$, we move between a neighbourhood $U$ of $x$ 
and a neighbourhood $V$ of $y$ frequently, 
so that there should be some 
other accumulation point in $\overline \unitdisc\setminus (U\cup V)$)
leading us to the previous case again.} 
and belong to $ \overline {\unitdisc}$. If 
$\lim_{t \to a^+} \gamma_I(t)$ belongs to
$\partial \unitdisc$, 
it must be either $q_1$ or $q_2$; if instead it is in $\unitdisc$, 
then we can always extend $\gamma_I$ in a neighbourhood of $a$ 
and find a larger interval $J\supset I$ on which $\gamma_I$ can be extended. 
A similar argument applies 
for $\lim_{t \to b^-} \gamma_I(t)$.
Let now 
$I_m=(a_m,b_m)$ be a maximal interval on which $\gamma_I$ is defined, 
so that, by maximality, the limits as $t \to a_m^+$ and $t \to b_m^-$ 
are $q_1$ and $q_2$, respectively.
We can then consider the closure $\overline I_m$ of $I_m$ 
and we have that $\gamma_{\overline I_m}(\overline I_m)$ 
is a curve in $\overline \unitdisc$ joining $q_1$ and $q_2$.
Thus we have proved that $\sigma_W:=\Sigma\cap \{w_1=W_1\}$ equals
$\Phi(\gamma_{\overline I_m}(\overline I_m))$.
In particular $\sigma_{W}$ is a curve in $\R^3$ contained 
in the plane $\{w_1 = W_1\}$
and connecting the points $\Phi(q_1)\in \Gamma$ 
and  $\Phi(q_2)\in \Gamma$. But we know that 
$\projthree(\Phi(q_1))=\projthree(\Phi(q_2))=
(W_1,\bar s,0)$, 
so $\projthree(\sigma_W)$ is a segment in $\doubledrectangle$  
with endpoints $(W_1,\bar s,0)$ and 
$(W_1,s^+,0)$ for some $s^+>\bar s$, and $s^+\geq W_2$. 
 In particular the whole 
segment ``below'' $\projthree(W)$, 
namely the one with endpoints $(W_1,\bar s,0)$ and $(W_1,W_2,0)$, belongs
 to $\projthree(\Sigma)$, and $\projthree(\Sigma)$ is then the subgraph of some function $\widetilde h$. 
As a remark, due to the symmetry of the curve $\Gamma$, we can 
assume $\widetilde h$ is symmetric with respect to $\{w_1=l\}$, namely $\widetilde h (\cdot)=\widetilde h (2l-\cdot)$.
 
 Now we show that $\widetilde h$ is convex. 
Assume it is not, and 
take two points $(t_1,\widetilde h(t_1),0), (t_2,\widetilde h(t_2),0)\in 
\doubledrectangle$,  $t_1<t_2$,
and a third point $(\terzopunto, \widetilde h(\terzopunto),0)$, 
with $t_1<\terzopunto<t_2$, which is strictly above the segment 
$l_{12}$ in $\doubledrectangle$ joining $(t_1,\widetilde h(t_1),0)$ 
and $(t_2,\widetilde h(t_2),0)$. Let $f:\R^3\rightarrow \R$ be a 
nonzero affine function\footnote{Take  
the signed distance from the plane.} 
vanishing on the plane passing through $l_{12}$ 
and orthogonal to $\{w_3=0\}$, and assume that 
$f$ is positive at $(\terzopunto,\widetilde h(\terzopunto),0)$. Let $\pointQ\in \Sigma$ be 
such that $\projthree(\pointQ)=(\terzopunto,\widetilde h(\terzopunto),0)$. 
Then $f\circ \Phi:
\unitdisc
\rightarrow \R$ is harmonic, and by the maximum principle there is a
continuous curve\footnote{The set $(f\circ\Phi)^{-1}((0,\infty))$ is open, and cannot have connected components not intersecting $\partial \unitdisc$, by harmonicity.} $\gamma_\pointQ$ in $\unitdisc$ 
joining $\Phi^{-1}(\pointQ)$ to $\partial \unitdisc$ such that 
$f\circ \Phi$ is always positive on $\gamma_\pointQ$. {}
 But now, the continuous
curve $\projthree\circ \Phi(\gamma_\pointQ)$ joins $(\terzopunto,\widetilde h(\terzopunto),0)$ 
to $\projthree(\Gamma)$ and remains, in $\doubledrectangle$, strictly above the segment 
$l_{12}$. This is a contradiction,
 because $\projthree\circ \Phi(\gamma_\pointQ)$ must be in the interior of
 the subgraph of $\widetilde h$.

\smallskip

Before passing to step 6, recall the definition of $\Gamma_0$ in 
\eqref{eq:Gamma_0}, and  observe
that the Jordan curve $\Gamma^+\cup\Gamma_0$ is the boundary of the disc-type surface $\dtsp$.

Let us denote by $U\subset K$ the connected component of $K\setminus \Gamma_0$ with boundary $\Gamma_0\cup (\{0\}\times [\bar s,1])\cup ([0,2l]\times \{\bar s\})\cup (\{2l\}\times [\bar s,1])$.

We are now in a position to show that $\Sigma^+$ admits
a non-parametric description over the plane $\{w_3=0\}$.

\medskip
\textit{Step 6: The disc-type surface $\dtsp$ 
can be written as a graph over the plane $\{w_3=0\}$ 
of a $W^{1,1}$ function 
$\widetilde \psi:U\rightarrow [0,+\infty)$}. 

At first we observe that if $\Sigma^+$ is not Cartesian with 
respect to $\{w_3=0\}$, then there is 
some point $\pointP\in \Sigma^+\setminus \partial \Sigma^+$ 
where the tangent plane to $\Sigma^+$ is vertical,
that is, it contains the 
line $\{\pointP+(0,0,w_3):\;w_3\in \R\}$.
This can be 
seen as follows: as shown in step 5,  
the intersection between $\Sigma^+$ and any plane $\{w_1=\textrm{cost}\}$,
${\rm cost}\in (0,2l)$, is a 
simple curve with endpoints in $\partial \Sigma^+$. If $\Sigma^+$ is not Cartesian, one of these curves $\gamma$ 
is not Cartesian, and then there is a point where 
the tangent vector to $\gamma$ is vertical. 
At such a point the tangent plane to $\Sigma^+$  
is vertical. 

\begin{itemize}
	\item[Claim:] If $\Pi$ is a vertical plane tangent to $\Sigma$, then there is at most one point where 
$\Pi$ and $\Sigma$ are tangent.
\end{itemize}
We use an argument similar to 
the one needed to prove Rado's Lemma \cite[Lemma 2, pag. 295]{Hil1}.
Assume $\Pi$ intersects the relative interior of $\Sigma$. It is easy to 
see that the intersection between $\Pi$ and the Jordan curve $\Gamma$ consists at most 
of four points\footnote{A vertical plane $\Pi$ 
intersects $K$ on a straight segment. In turn, this segment intersects 
$\partial K$ in two points. If $\Pi$ intersects $\Gamma$ in a point $(W_1,W_2,W_3)$, then $(W_1,W_2,0)\in \partial K$. Moreover,
$\Pi$ intersects $\Gamma$ also at $(W_1,W_2,-W_3)$. Thus, the points of intersection are at most four. The degenerate cases in which $\Pi$ contains a full $\mathcal H^1$-measured part of $\Gamma$ are excluded by this analysis, because in these cases $\Pi$ does not intersect the interior of $\Sigma$. Instead, the cases in which the intersection consists of $2$ or $3$ points are easier to treat, and we detail only the $4$-points case (notice that by the geometry of $\Gamma$, the case of  $3$ points occurs when this plane is tangent to $\Gamma$ at one of the points $(0,1,0)$ or $(2l,1,0)$).} $p_i$, $i=1,2,3,4$.
Let $f$ be a linear function on $\R^3$ vanishing on $\Pi$.
Then  
 $f\circ \Phi$ is harmonic  in $\unitdisc$ 
and continuous in $\overline \unitdisc$; in addition, it vanishes
at $\{p_i,\;i=1,2,3,4\}$, and alternates its sign on the relatively
open four arcs
$\overline{p_i p_{i+1}}$
on $\partial \unitdisc$
with endpoints $p_i$.
With no loss of generality, we may assume $f\circ\Phi>0$ on $\overline{p_1p_2}$ and $\overline{p_3p_4}$. 
By harmonicity of $f\circ\Phi$, 
any connected component of the region $\{x\in \overline \unitdisc:f\circ\Phi(x)>0\}$  must contain part of  $\overline{p_1p_2}$ or $\overline{p_3p_4}$, so that we deduce that these connected components are at most two. 

Assume now by contradiction that there are two distinct 
points $\pointP$ and $\pointQ$ of $\Sigma$ such that $\Pi$ 
is tangent to $\Sigma$ at $\pointP$ and $\pointQ$.   
Since $f\circ\Phi$ has null differential at $\Phi^{-1}(\pointP)$ 
and $\Phi^{-1}(\pointQ)$, the set $\{f\circ \Phi=0\}$, 
in a neighbourhood of $\Phi^{-1}(\pointP)$, consists of $2m_p$ analytic 
curves crossing at $\Phi^{-1}(\pointP)$, whereas  in a neighbourhood of 
$\Phi^{-1}(\pointQ)$, it consists of $2m_q$ analytic curves crossing at 
$\Phi^{-1}(\pointQ)$. Therefore, in a neighbourhood of $\Phi^{-1}(\pointP)$, 
the set $\{f\circ \Phi>0\}$ counts at least $2$ open regions 
(and similarly at $\Phi^{-1}(\pointQ)$). Let us call 
$A_1$ and $A_2$ two of these regions
around $\Phi^{-1}(\pointP)$,
and 
$B_1$, $B_2$ two of these regions
around 
$\Phi^{-1}(\pointQ)$. 
By harmonicity each $A_i$ and $B_i$ must be connected to one of the 
arcs   $\overline{p_1p_2}$ or $\overline{p_3p_4}$. 
Hence some of these regions must belong to the same connected component
of $\{f\circ\Phi>0\}$. Then we are reduced to two following
cases (see Fig. \eqref{fig:twocases}): 
\begin{itemize}
	\item[(Case A)] $A_1$ and $A_2$ belong to the same connected 
component, 
say the one containing $\overline{p_1p_2}$. Hence we can construct two disjoint curves in $\{f\circ \Phi>0\}$, both joining $\Phi^{-1}(\mathcal P)$ to a point in   $\overline{p_1p_2}$, emanating from $\Phi^{-1}(\mathcal P)$, one in 
region $A_1$ and one in region $A_2$. This contradicts the 
maximum principle, because these two curves would enclose a region 
where $f\circ\Phi$ takes also negative values, whereas its boundary is in $\{f\circ\Phi>0\}$.
	\item[(Case B)] $A_1$ and $B_1$ are joined 
to $\overline{p_1p_2}$ and $A_2$ and $B_2$ are joined to $\overline{p_3p_4}$. 
In this case we can construct four curves in $\{f\circ\Phi>0\}$:  $\sigma_1$ and $\sigma_2$ emanating from $\Phi^{-1}(\mathcal P)$ in regions $A_1$ and $A_2$ and reaching $\overline{p_1p_2}$ and $\overline{p_3p_4}$, respectively; $\beta_1$ and $\beta_2$ emanating from $\Phi^{-1}(\mathcal Q)$ in regions $B_1$ and $B_2$ and reaching $\overline{p_1p_2}$ and $\overline{p_3p_4}$, respectively.
	The region enclosed between these $4$ curves has boundary contained in $\{f\circ\Phi>0\}$ and, inside it, necessarily 
the function $f\circ\Phi$ takes also negative values, again in contrast with the 
maximum principle.
\end{itemize} 
From the above discussion our claim follows.

We are now ready to conclude the proof of step 6:
suppose by contradiction  that $\Sigma^+$ is not Cartesian
with respect to $\{w_3=0\}$, and take a point $P^+ 
\in \Sigma^+ \setminus \Gamma$
where the tangent plane $\Pi$ to $\Sigma^+$ at $P^+$ is vertical. By symmetry of $\Sigma$, the point $P^-$, 
defined as the symmetric of $P^+$ with respect to the 
rectangle $\doubledrectangle$, belongs to $\Sigma^-$, and the tangent plane to $\Sigma^-$ at $P^-$ is the same plane $\Pi$. This contradicts the claim.
We eventually observe that $\widetilde \psi$ is analytic on the subgraph of $\widetilde h$, since its graph is $\Sigma^+$. We conclude that $\widetilde \psi$ belongs to $W^{1,1}(SG_{\widetilde h})$, since  
its total variation is bounded by the area of its graph, which is finite.

\nada{
Setting $\widetilde \dts:=\partial \mathbb 
S_{cl}(E)\setminus (\partial K\times \R)$, we see that
$\widetilde \dts$ is a disc-type 
surface whose boundary is $\Gamma$ (here we use 
that $\Gamma$ is symmetric and that $\Gamma^+$ and $\Gamma^-$ are Cartesian, 
with respect to $\{w_3=0\}$). Indeed, $\widetilde \Sigma$ is 
a Cartesian surface, graph of some function defined on the subgraph of 
$\widetilde h$, which has the topology of the disc. 
	
We deduce
\begin{equation}\label{strict_ineq}
\mathcal H^2(\widetilde \dts)<\mathcal H^2(\dts),
\end{equation}
 contradicting the minimality of $\dts$.  
This proves our claim, and
therefore there exists 
$$\widetilde \psi\in BV(U,[0,+\infty))$$
such that its 
graph over $\overline{U}$ coincides with $\dtsp$.
Hence we can write
\begin{align}
 \mathcal H^2(\dtsp)= \areaonecod(\widetilde \psi, \overline U),
\end{align}
and from \eqref{ineq_withF} we conclude
\begin{align}\label{ineq_withF2}
 \areaonecod(\widetilde \psi, \overline U)+
\areaonecod(\psi, R\setminus K)\leq \FB(h,\psi).
\end{align}
}

\begin{figure}
	\begin{center}
		\includegraphics[width=0.7\textwidth]{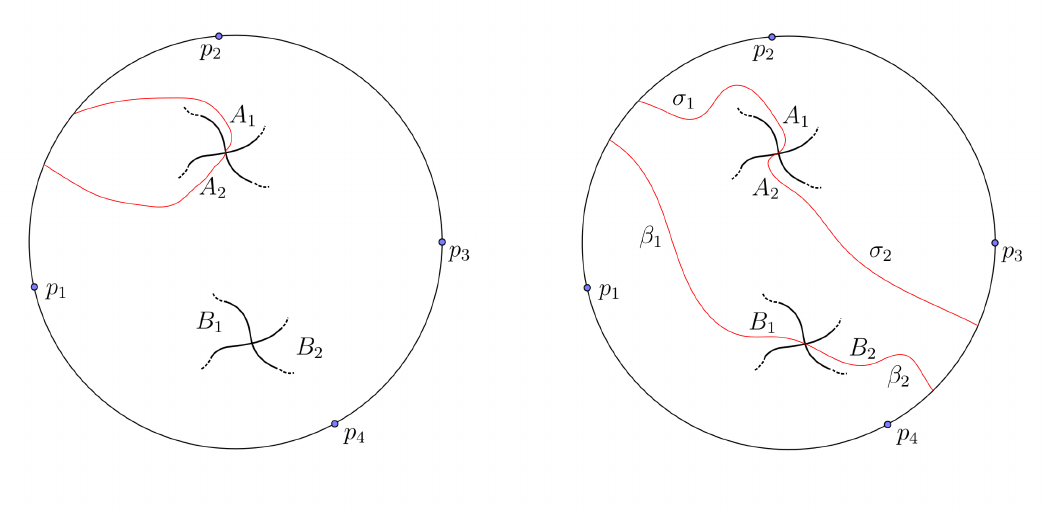}
		\caption{On the left it is represented 
case A in step 6 of the proof of Theorem \ref{teo:boundary_regularity}. The point $\Phi^{-1}(\mathcal P)$ is in the cross where the two emphasized paths start from. These curves stand in the region $\{f\circ\Phi>0\}$ and join $\Phi^{-1}(\mathcal P)$ with the 
arc $\overline{p_1p_2} \subset \partial \unitdisc$. The picture on the right represents instead case B. The two cross points are $\Phi^{-1}(\mathcal P)$ and $\Phi^{-1}(\mathcal Q)$ and the paths $\sigma_1$, $\sigma_2$, $\beta_1$, $\beta_2$ are depicted.
		}
		\label{fig:twocases}
	\end{center}
\end{figure}

\medskip

\textit{Step 7: the pair 
$(\widetilde h, \widetilde \psi)\in X_{2\longR}^{\rm conv}$.}

We recall that in step 5 we 
proved that $\widetilde h$ is convex and 
$\widetilde h(\cdot)=\widetilde h(2l-\cdot)$, \textit{i.e.}  $\widetilde h\in \mathcal H_{2l}$.
Furthermore $\Sigma^+$ is the graph of $\widetilde \psi$, and its projection on the plane $\{w_3=0\}$ is the subgraph of $\widetilde h$. It follows that the area of the graph of $\widetilde \psi$ 
is exactly the area of $\Sigma^+$ upon $SG_{\widetilde h}$. Let us also recall the 
$W^{1,1}$ regularity of $\widetilde \psi$ proved in 
step 6.
Setting 
$$
\widetilde \psi:=\psi \qquad{\rm in}~ \doubledrectangle\setminus K,
$$
 we infer $(\widetilde h,\widetilde \psi)\in X_{2\longR}^{\rm conv}$.
Since $\psi$ is analytic, this construction turns 
out to be independent of the choice of $\overline s$.

\medskip
\textit{Step 8: Conclusion of the proof.} 

From \eqref{ineq_withF} we deduce
\begin{align}
 \FB(h,\psi)\geq \mathcal H^2(\dtsp)+ \areaonecod(\psi, \doubledrectangle\setminus K)= \FB(\widetilde h,\widetilde \psi),
\end{align}
where the last equality follows from the fact that $\widetilde \psi$ 
is continuous on $\partial_D\doubledrectangle$. Hence, also $(\widetilde h,\widetilde \psi)$ is a minimizer for $\FB$.
Now, we show (ii) and (iii), namely 
that $\widetilde \psi$ is continuous and equals $0$ on $G_{\widetilde h}$. 
Indeed $\dts=\dtsp\cup \dtsm$ is 
analytic, 
 hence the graph of $\widetilde h$ coincide with the intersection of the analytic surface $\Sigma$ with the plane $\{w_3=0\}$ (which is not tangent to $\Sigma$);
 it follows that $\widetilde h$ is continuous. 
Moreover we know that $\widetilde \psi$ is smooth in 
$SG_{\widetilde h}$. 
If its  trace $\widetilde \psi^+$ 
on the boundary of $SG_{\widetilde h}$ is strictly positive somewhere, 
say at $Z\in  G_{\widetilde h}\times \{0\}$,
we infer that the vertical segment defined as
$$\{(Z_1,Z_2,w_3):|w_3|\in (0,\widetilde \psi^+(w_1,w_2))\},$$
is contained in $\Sigma\cap\{w_1=Z_1\}$, 
which is an analytic curve. This would imply
 that $\Sigma\cap\{w_1=Z_1\}$ is contained in the straight line $(Z_1,Z_2)
\times \R$, which is a contradiction, because $(Z_1,-1,0)\in 
\Sigma\cap\{w_1=Z_1\}$, and $Z_2>-1$.
% positive $\mathcal H^2$-measure and 
%cannot have zero mean curvature (the only case in which 
%its mean curvature vanishes is when $\widetilde h$ is linear, 
%but in this case $\dtsp$ must be contained in a plane 
%containing $G_{\widetilde h}$ which is impossible,
%since $\Gamma^+$ is not).
We conclude $\widetilde \psi^+=0$ on $G_{\widetilde h}$.

Finally, let us prove (i), 
i.e., that $L_{\widetilde h} = \emptyset$. For, if 
not, the vertical part of $\dtsp$ obtained on $L_{\widetilde h}$ 
is flat and then, by analyticity, also $\dtsp$ is, a contradiction.
This completes the proof.
\qed 

\medskip

A direct consequence of Theorem \ref{teo:boundary_regularity} is the following
which,
coupled with Corollary \ref{cor:smoothness_of_h} and its 
subsequent discussion,
 concludes the proof of Theorem \ref{teo:main1}.
\begin{cor}\label{cor_reg}
	Let $(h,\psi) \in X_{2\longR}^{{\rm conv}}$ be a solution 
of \eqref{eq:B_intro}. Then,  either $(h,\psi)$ is degenerate, i.e. $h\equiv-1$, 
or:
\begin{itemize}
\item[(i)]
$h(0)=1=h(2l)$, 
\item[(ii)]
$\psi$ is continuous up
to the boundary of $SG_{h}$,
\item[(iii)]
$\psi = 0$ on $G_h$.
\end{itemize}
\end{cor}
\begin{proof}
As in the proof of Theorem \ref{teo:boundary_regularity}, if $(h,\psi)$ is not degenerate, then $(h,\psi)$ is a solution as in Lemma \ref{lemma_reg2_bis} and we can choose $\overline s\in (-1,h(l))$. 
In step 7 of that proof we found that $(\widetilde h,\widetilde \psi)$ is another minimizer, with $\widetilde \psi$ 
coinciding with $\psi$ on $\doubledrectangle
\setminus K$. 
Now,
we claim that $h = \widetilde h$ and $\psi = \widetilde \psi$.
By analyticity of both $\psi$ and $\widetilde \psi$ in 
their domain $SG_h$ and $SG_{\widetilde h}$
of definition respectively, 
they must coincide on $SG_h\cap SG_{\widetilde h} $. Hence, if $\overline w_1\in (0,2l)$ is such that $\widetilde h(\overline w_1)<h(\overline w_1)$, we deduce that $\psi(\overline w_1,\widetilde h(\overline w_1))=\widetilde \psi(\overline w_1,\widetilde h(\overline w_1))$ contradicting the 
maximum principle because $(\overline w_1,\widetilde h(\overline w_1))\in SG_h$. Therefore, necessarily $h\leq \widetilde h$. This implies that the trace of $\psi$ on $G_h$ coincides with the restriction of $\widetilde \psi$ on $G_{h}$, and thus this trace is continuous. If, by contradiction $h(\overline w_1)<\widetilde h(\overline w_1)$ for some $\overline w_1\in (0,2l)$, we find a contradiction as follows: Let $(a,b)$ be 
a maximal interval containing $\overline w_1$ 
where $h<\widetilde h$ on it. Suppose for simplicity that $(a,b)=(0,2l)$ (otherwise a similar argument applies).

We consider the curve $\Gamma^+$ obtained by glueing the graph of $\psi=\widetilde \psi$ over $G_h$ with the graph of $\varphi$ on the two segments $L_h$, and we consider also $\Gamma^-$, the symmetric of $\Gamma^+$ with respect to the plane $\{w_3=0\}$. Since both $\psi$ and $\varphi$ are strictly positive on these domains, the curve $\Gamma:=\Gamma^+\cup \Gamma^-$ is a Jordan curve. The surface 
$$S:=\{w\in \R^3:(w_1,w_2)\in G_h,\;|w_3|\leq \psi(w_1,w_2)\}\cup\{w\in \R^3:(w_1,w_2)\in L_h,\;|w_3|\leq \varphi(w_1,w_2)\}$$
is a disc-type surface spanning $\Gamma$, and thus $\mathcal H^2(S)\geq \mathcal H^2(\Sigma_\Gamma)$ 
where $\Sigma_\Gamma$ is a disc-type 
solution of the Plateau problem spanning
$\Gamma$. By definition of $\widetilde \psi$, its graph $\mathcal G_{\widetilde \psi}$ over $SG_{\widetilde h}\setminus SG_h$ 
enjoys the property that, denoting by $\mathcal G^-_{\widetilde \psi}$ its symmetric with respect to the plane $\{w_3=0\}$, 
the surface $\widetilde \Sigma:=\mathcal G_{\widetilde \psi}\cup \mathcal G_{\widetilde \psi}^-$ is a solution to the Plateau problem for discs spanning $\Gamma$. Hence we deduce that $\mathcal H^2(S)\geq \mathcal H^2(\widetilde\Sigma)$. Since however $(h,\psi)$ and $(\widetilde h,\widetilde \psi)$ are both minimizers of $\mathcal F_{2l}$, 
the same argument in Step 8 of Theorem 
\ref{teo:boundary_regularity}
implies $\mathcal H^2(S)= \mathcal H^2(\widetilde\Sigma)$, and 
$S$ is a solution to the Plateau problem for discs spanning $\Gamma$. 
However, unless $h\equiv1$, this contradicts the strong maximum principle, because $\Gamma$ is a non-planar curve contained in the boundary of the convex set $$C:=\{w\in \R^3:w_2>h(w_1)\},$$ and so the interior of $S$ cannot lie on $\partial C$. 
This contradiction 
leads us to our claim, namely  $h=\widetilde h$, and $\psi=\widetilde \psi$.
\end{proof}

Now, we discuss the smoothness of $h$:
\begin{cor}\label{cor:smoothness_of_h}
	If a solution $(h,\psi)$ of \eqref{eq:B_intro} is not degenerate (i.e. $h$ is not constantly $-1$), then $h\in C([0,2l])
$
and it is analytic in $(0,2l)$.
\end{cor}
\begin{proof}
	The continuity of $h$ at $0$ and $2l$ is clear since $L_h=\varnothing$.
	Going back to Step 4 of the proof of Theorem 
\ref{teo:boundary_regularity}, we have seen that the graph $G_h$ of $h$ 
coincides with the curve $\Gamma_0=\Sigma\cap\{w_3=0\}$ 
(see \eqref{eq:Gamma_0}). Let $t_0\in (0,2l)$, and let $\puntone=
(t_0,h(t_0),0)\in \Sigma\cap \{w_3=0\}$. Let 
$\gamma_0\subset \overline\unitdisc$
be defined as $\gamma_0:=\Phi^{-1}(\Sigma\cap \{w_3=0\})=\{x\in \overline \unitdisc:\Phi_3(x)=0\}$, 
which is a simple curve in $\unitdisc$ connecting two points on 
$\partial \unitdisc$, and let $\puntino:=\Phi^{-1}(\puntone)\in \gamma_0$. 
Setting $\partial_i = \frac{\partial}{\partial w_i}$,
we have that $\partial_{1} \Phi(\puntino)$ and $\partial_2\Phi(\puntino)$ 
are distinct vectors generating the tangent plane to $\Sigma $ at $\puntone$, 
which is a vertical plane; hence $\partial_{1} \Phi_3(\puntino)$ 
and $\partial_2\Phi_3(\puntino)$ cannot be both $0$ (say, $\partial_{2} 
\Phi_3(\puntino)\neq0$). Therefore,
by the implicit function theorem, in a neighborhood of $\puntino$, 
$\gamma_0$ can be parametrized by a function $\sigma:(-\delta,\delta)
\rightarrow \gamma_0$, $\sigma(s)=(s,f(s))$ with $f'(s)=-\frac{\partial_{1} \Phi_3(s,f(s))}{\partial_{2} \Phi_3(s,f(s))}$; 
as $\Phi$ is analytic in $\unitdisc$, 
we deduce that $f$, and therefore $\sigma$, are analytic in a neighborhood of $0$.
	Now 
	$s\mapsto\Phi(\sigma(s))$ parametrizes $\Gamma_0$ in a 
neighborhood of $\puntone$. As we know that $\Gamma_0$ is the graph of $h$, we see that $\frac{d}{ds}\Phi_1(\sigma(s))$ is 
non-zero in a neighborhood of $0$. Defining the parameter
	$$t(s):=t_0+\int_0^s|\nabla \Phi_1(\sigma(s))\sigma'(s)|ds,$$
	if $s(t)$ denotes its inverse, we have
	$$\frac{d}{dt}s(t)=\frac{1}{|\nabla \Phi_1(\sigma(s(t)))\sigma'(s(t))|},$$
	and $s(\cdot)$ is analytic in a neighborhood of $t_0$.
	Then we have
	$h(t)=\Phi_2(\sigma(s(t)))$, which is the composition of 
analytic maps, hence
analytic  in a neighborhood of $t_0$. As this holds for all $t_0\in (0,2l)$, the assertion follows.
\end{proof}

\medskip

To conclude the proof of 
Theorem \ref{teo:main1},
it remains to show (2iv). 
The pair $(h \equiv 1,  \varphi)$ , where 
the function $\varphi$ is as in \eqref{eq:varphi},
is one of the competitors for
problem \eqref{eq:B_intro} (notice that $\varphi$ attains the boundary condition); in addition,
its subgraph is strictly convex 
(see Fig. \ref{fig:graphofphi}), hence\footnote{As already observed, 
the minimal surface $\Sigma^+$ is the graph of $\psi=\widetilde \psi$, and it must be contained in the convex envelope of $\Gamma$, \textit{i.e.}, inside the subgraph of $\varphi$.}  necessarily
$\psi\leq \widehat \varphi$ in $\overline R_{2\longR}$, where 
we have taken  $\psi=\widetilde \psi$, the solution given by Theorem 
\ref{teo:boundary_regularity}.

Eventually, the strict inequality in Theorem \ref{teo:main1} (2iv)
is a consequence of the strong maximum principle: indeed, 
points in $\Sigma\setminus \partial\Sigma$  are always strictly inside the convex hull of $\partial\Sigma$, with the only exception when $\partial\Sigma$ is planar 
(see \cite[pag 63, section 70]{Nitsche:89}); so that points of $\Sigma^+\setminus \partial\Sigma$ are strictly inside the graph $G_{\varphi}$ of $\varphi$ (that is  half of the lateral boundary of a cylinder).
\end{proof}

Now, we point out another consequence of 
Theorem \ref{teo:boundary_regularity}, {which gives
the proof of Theorem \ref{teo:main2}}.
Let $G_w$ be the graph in $\doubledrectangle$ of a 
function $w\in C([0,2l],(-1,1])$ such that $w(0)=w(2l)=1$,
 and consider the curve $\Gamma_w$ obtained by concatenation 
of $G_w$ with the graph of $\varphi$ over $\partial_D\doubledrectangle$. 
\begin{cor}\label{main_cor}
We have
\begin{align}\label{equiv_plateau}
\mathcal F_{2l}(h,\psi)=\inf\mathcal P_{\Gamma_w}(X_{{\rm min}}),
\end{align}
where $(h,\psi)\in X_{2l}^{\rm{conv}}$ 
is a minimizer of $\mathcal F_{2l}$, $X_{{\rm min}}$ 
is a parametrization of a disc-type 
area-mininizing solution of 
the Plateau problem spanning  $\Gamma_w$ (see \eqref{plateau}), 
and the infimum is computed over all functions $w$ as above.
\end{cor}
The proof of this corollary
 can be achieved by adapting the proof of Theorem \ref{teo:boundary_regularity}, 
which shows that the solution to the 
Plateau problem in \eqref{equiv_plateau} is Cartesian and the optimal $w$ is convex.
{
%We finally show that any minimizer as in Theorem \ref{Thm:existenceofminimizer} is either degenerate (i.e. $h^\star=-1$) or it satisfies the thesis of Theorem \ref{teo:boundary_regularity}.
%{\color{red}frase incompleta}

%%%%%%%%%%%%%%%%%%%%%%%%%%%%%%%%%%%%%%%%%%%%%%%%%%%%%%%%%%%%%%%%%%%%%%%
\section*{Acknowledgements}
%%%%%%%%%%%%%%%%%%%%%%%%%%%%%%%%%%%%%%%%%%%%%%%%%%%%%%%%%%%%%%%%%%%%%%%
The first and third authors acknowledge the support
of the INDAM/GNAMPA.
The first two authors are
grateful to ICTP (Trieste), where part of this paper was written.
 The first and third authors also acknowledge the partial financial support of the F-cur project number 2262-2022-SR-CONRICMIUR$_{-}$PC-FCUR2022$_{-}002$ of the University of Siena, and the of the PRIN project 2022PJ9EFL "Geometric Measure Theory: Structure of Singular
Measures, Regularity Theory and Applications in the Calculus of Variations'', PNRR Italia Domani, funded
by the European Union via the program NextGenerationEU, CUP B53D23009400006.

\end{document}